\documentclass[a4paper,12pt,headsepline,cleardoubleempty,reqno,bibtotoc,tablecaptionabove]{amsart}
\usepackage[utf8]{inputenc}
\author{David Wallauch}
\address{Universität Wien, Fakultät für Mathematik,
  Oskar-Morgenstern-Platz 1, 1090 Vienna, Austria}
\email{david.wallauch@univie.ac.at}
\thanks{This work was supported by the Austrian Science Fund FWF,
  Projects P 30076: ``Self-similar blowup in dispersive wave equations''
  and P 34560: ``Stable blowup in supercritical wave equations''.
  The author would also like to thank Prof. Roland Donninger for valuable tips and helpful discussions}
\title{Strichartz estimates and Blowup stability for energy critical nonlinear wave equations}
\usepackage{amsmath,amsfonts,amssymb,amsthm,empheq}
\usepackage{xcolor}
\usepackage{hyperref}
\usepackage{array}
\usepackage{fullpage}
\usepackage{graphicx}
\numberwithin{equation}{section}

\newcommand{\C}{\mathbb{C}}

\newcommand{\R}{\mathbb{R}}

\newtheorem{thm}{Theorem}[section]
\newtheorem{defi}{Definition}[section]

\newtheorem{lem}{Lemma}[section]
\renewcommand{\O}{\mathcal{O}}

\newcommand{\rg}{\textup{\textbf{rg}}}
\renewcommand{\ker}{\textup{\textbf{ker}}}

\newcommand{\hfh}{\textup{\textbf{h}}}

\newcommand{\Cf}{\textup{\textbf{C}}}

\newcommand{\X}{\mathcal{X}}

\newcommand{\Nf}{\textup{\textbf{N}}}
\newcommand{\Lf}{\textup{\textbf{L}}}

\newcommand{\Af}{\textup{\textbf{A}}}
\newcommand{\Ef}{\textup{\textbf{E}}}
\newcommand{\hf}{\textup{\textbf{h}}}

\newcommand{\K}{\textup{\textbf{K}}}
\newcommand{\Sf}{\textup{\textbf{S}}}
\newcommand{\I}{\textup{\textbf{I}}}
\newcommand{\Uf}{\textup{\textbf{U}}}
\newcommand{\uf}{\textup{\textbf{u}}}
\newcommand{\vf}{\textup{\textbf{v}}}
\newcommand{\Rf}{\textup{\textbf{R}}}
\newcommand{\gf}{\textup{\textbf{g}}}

\newcommand{\ff}{\textup{\textbf{f}}}
\newcommand{\Pf}{\textup{\textbf{P}}}

\renewcommand{\Re}{\operatorname{Re}}
\renewcommand{\Im}{\operatorname{Im}}
\newcommand{\B}{\mathbb{B}}

\renewcommand{\H}{\mathcal{H}}

\begin{document}

      \begin{abstract}
        We prove Strichartz estimates in similarity coordinates for the radial wave equation with a self similar potential in dimensions $d\geq 3$. As an application of these, we establish the asymptotic stability of the ODE blowup profile of the energy critical radial nonlinear wave equation for $3\leq d\leq 6$.
      \end{abstract}
      \maketitle
\section{Introduction}
We consider the focusing energy critical nonlinear wave equation 
\begin{equation}\label{Eq:startingeq}
\begin{cases}\left(\partial_t^2-\Delta\right)u(t,x)=u(t,x)|u(t,x)|^{\frac{4}{d-2}}
\\
u[0]=(f,g)
\end{cases}
\end{equation}
for a fixed dimension $d\geq 3$ and $u:I\times \R^{d}\to \R$, for some interval $0\in I\subset\R$ and $u[t]=(u(t,.),\partial_t u(t,.))$.
When studying Eq.~\eqref{Eq:startingeq} in a low regularity setting, one can of course not hope for classical solutions and instead has to resort to a different notion of a solution. For Eq. \eqref{Eq:startingeq} this can be achieved by using Duhamel's formula,
\begin{align} \label{notion solution}
u(t,.)=\cos(t|\nabla|)f+\frac{\sin(t|\nabla|)}{|\nabla|}g+\int_0^t\frac{\sin((t-s)|\nabla|)}{|\nabla|}u(s,.)|u(s,.)|^{\frac{4}{d-2}} ds,
\end{align}
where  $\cos(t|\nabla|)$ and $\frac{\sin(t|\nabla|)}{|\nabla|}$ are the standard wave propagators. 
This weak formulation is a sensible expression for $(f,g)\in \dot{H}^1\times L^2(\R^d)$ and moreover, it is known that Eq. \eqref{Eq:startingeq} is locally well-posed in this regularity class (see \cite{LinSog95,BulCzuLiPav13}). Furthermore, this is the lowest possible regularity in which the equation is not ill-posed.  The key difficulty to prove local well-posedness of Eq. \eqref{Eq:startingeq} stems from the fact that the nonlinearity cannot be controlled by Sobolev embedding, as this would require the embedding $\dot{H^1}(\R^d)\hookrightarrow L^{\frac{2d+4}{d-2}}(\R^d)$, which evidently does not hold.
In dimensions smaller than $7$, this can be remedied by employing Strichartz estimates.
These are spacetime estimates of the form
\begin{align*}
\left\|\frac{\sin(t|\nabla|)}{|\nabla|}g \right\|_{L^{\frac{d+2}{d-2}}_t(\R)L^{\frac{2d+4}{d-2}}(\R^d)}\lesssim \|g\|_{L^2(\R^d)}
\end{align*}
and variants thereof. Loosely, the improved control given by these estimates, is a consequence of exploiting the dispersive nature of Eq. \eqref{Eq:startingeq}. \\
One of the most intriguing features of Eq. \eqref{Eq:startingeq} is the occurrence of singularities in finite time. An example of this phenomenon is given by the explicit solution
\[
u^T(t):=c_d(T-t)^{\frac{2-d}{2}}
\quad c_d=\left(\frac{d(d-2)}{4}\right)^{\frac{d-2}{4}},
\]
where $T>0$ can be freely choosen. To investigate whether this solution is an isolated example of this kind of blowup or if it actually plays a role in more generic evolutions, one has to study the stability of this family of solutions. 
To do so, we restrict ourselves to radial initial data for which Eq. \eqref{Eq:startingeq}
turns into
\begin{equation}\label{Eq:startingeq2}
\begin{cases}\left(\partial_t^2-\partial_r^2-\frac{d-1}{r}\partial_r\right)u(t,r)=u(t,r)|u(t,r)|^{\frac{4}{d-2}}
\\
u[0]=(f,g),
\end{cases}
\end{equation} 
for $r=|x|$ and where we identify functions with their radial representative.
Moreover,  we employ the finite speed of propagation property thanks to which we can study Eq. \eqref{Eq:startingeq2} in the backwards lightcone $\Gamma_T:=\{(t,r): t\in [0,T), r \in [0, T-t]\}.$
For this domain, especially useful coordinates are given by the similarity coordinates
\begin{align*}
\rho:=\frac{r}{T-t},\qquad\tau:=-\log(T-t)+\log(T).
\end{align*}
 In these coordinates we then linearise the nonlinearity around the blowup solution $u^T$
and set $\psi(\tau,\rho)=(Te^{-\tau})^{\frac{d-2}{2}}u(T-Te^{-\tau}, Te^{-\tau}\rho)$
and
\[\psi_1(\tau,\rho)=\psi(\tau,\rho),\qquad \psi_2(\tau,\rho)=\partial_\tau\psi(\tau,\rho)+\rho\partial_\rho\psi(\tau,\rho)+\frac{d-2}{2}\psi(\tau,\rho) \]
 to obtain an abstract evolution equation of the form
 \[\partial_\tau \Psi(\tau)=\widehat{\Lf} \Psi(\tau)+ \Nf(\Psi(\tau)),\]
 where $\widehat{\Lf}$ is a linear spatial differential operator.
 More precisely, 
 $\widehat{\Lf}$ formally takes the form 
 \begin{align*}
\widehat{\Lf} \begin{pmatrix}
u_1(\rho)\\
u_2(\rho)
\end{pmatrix}
=\begin{pmatrix}
-\rho u_1'(\rho)-\frac{d-2}{2}u_1(\rho)+u_2(\rho)\\
u_1''(\rho)+\frac{d-1}{\rho}u_1'(\rho)-\rho u_2'(\rho)-\frac{d}{2}u_2(\rho)
\end{pmatrix}+
\begin{pmatrix}
0\\
\frac{2d+d^2}{4}u_1(\rho)
\
\end{pmatrix}.
\end{align*}
 Our first theorem are the following Strichartz estimates.
\begin{thm}\label{thm:strichartz}
Let $d\geq 3$ be a fixed natural number and define the domain of $\widehat{\Lf}$ to be $C^2\times C^1(\overline{\B^d_1})$. Then $\widehat{\Lf}$ is closable and its closure $\Lf$ generates a semigroup $\Sf$ on $H^1\times L^2(\B^d_1)$ such that the following holds. There exists a one dimensional subspace $\Uf\subset H^1\times L^2(\B^d_1)$ and a bounded projection $\Pf:H^1\times L^2(\B^d_1)\to \Uf$ such that 
\begin{align*}
\|\Sf(\tau) \Pf \ff \|_{H^1\times L^2 (\B^d_1)}\lesssim e^{\tau}\|\ff\|_{H^1\times L^2(\B^d_1)}
\end{align*}
holds for all $\ff \in H^1\times L^2(\B^d_1)$ and all $\tau \geq 0$.
Moreover,
for $p\in [2,\infty]$ and $ q\in [\frac{2d}{d-2},\frac{2d}{d-3}]$ with
$\frac{1}{p}+\frac{d}{q}=\frac{d}{2}-1$,
 the bound 
\begin{align*}
\|[\Sf(\tau)(\I-\Pf)\ff]_1\|_{L^p_\tau(\R^+)L^q(\B^d_1)}\lesssim \|(\I-\Pf)\ff\|_{H^1\times L^2(\B^d_1)}
\end{align*}
holds for all $\ff \in H^1\times L^2(\B^d_1)$.
Additionally, the inhomogenous estimate 
\begin{align*}
\left\|\int_0^\tau[\Sf(\tau-\sigma)(\I-\Pf)\hf(\sigma,.)]_1 d \sigma\right\|_{L^p_\tau(I)L^q(\B^d_1)}\lesssim \|(\I-\Pf)\hf(\tau,.)\|_{L^1_\tau(I)H^1\times L^2(\B^d_1)}
\end{align*}
holds for all $\hf \in C([0,\infty),H^1\times L^2(\B^d_1))\cap L^1([0,\infty),H^1\times L^2(\B^d_1))$ and all intervals $I=[0,\tau_0)\subset[0,\infty)$.
\end{thm}
Before we come to our second theorem, we would like to make the following remarks.
\begin{itemize}
\item
The one dimensional unstable subspace stems from linearising around $u^T$. This produces the eigenvalue $1$, which corresponds to the time translation symmetry of the equation and hence is not a real instability.
\item
In Cartesian coordinates the potential obtained from the linearisation has a singularity at the tip of the lightcone $\Gamma_T$. Our result therefore shows Strichartz estimates for a wave equation in a lightcone with a potential that exhibits singular behavior at the lightcone's tip.
\item
The method we employ is very robust and can easily be adopted to other potentials provided one has enough spectral information.
\end{itemize}

These Strichartz estimates are the main tool used to investigate the blowup stability of $u^T$ in the critical topology, i.e., to prove the following result.
\begin{thm}\label{thm: stability}
Let $3\leq d\leq 6$ be a fixed natural number. Then there exist constants $M>1$ and $\delta_0>0$ such that for $\delta \in (0,\delta_0)$ the following holds. Let $(f,g)\in H^1\times L^2(\B^d_{1+\delta})$ be such that
\begin{align*}
\|(f,g)-u^1[0]\|_{H^1\times L^2(\B^d_{1+\delta})}\leq \frac{\delta}{M}.
\end{align*} 
Then there exists a unique solution $u$ to Eq. \eqref{Eq:startingeq2} and a $T$ in $[1-\delta, 1+\delta]$ such that
\begin{align}\label{estimatethm}
 \int_0^T\|u(t,.)-u^T(t,.)\|_{L^{\frac{2d}{d-3}}(\B^d_{T-t})}^2 dt\leq \delta^2.
\end{align}
\end{thm}
Once more we would like to make some remarks.
\begin{itemize}
\item
One readily computes that 
$$\|u^T(t,.)\|_{L^{\frac{2d}{d-3}}(\B^d_{T-t})}^2\simeq (T-t)^{-1}.
$$
Hence, for the estimate \eqref{estimatethm} to hold, $u$ has to exhibit the same blowup behavior as $u^T$ modulo a small error.  Consequently, Theorem \ref{thm: stability} states that there is an open ball around $u^1[0]$ in the energy topology such that data inside that ball leads to the ODE type blowup. Observe, however, that the actual blowup time gets slightly shifted in general. This shift is a consequence of the time translation symmetry of Eq. \eqref{Eq:startingeq2}.
\item
The topology used is optimal in that one cannot lower the regularity assumptions on the initial data as Eq.~\eqref{Eq:startingeq2} is ill-posed below $H^1\times L^2$.
\item
The restriction $d\leq 6$ is due to the fact, that starting from dimension $7$ one cannot solely rely on the established $L^p L^q$ spacetime estimates to control the nonlinearity. 
\end{itemize}
\subsection{Outline of the proof of Theorem \ref{thm:strichartz}}
As the main body of this work is concerned with proving Theorem \ref{thm:strichartz}, we highlight the key steps. We also emphasize that the strategy employed to tackle this problem builds on the techniques devised in \cite{Don17,DonRao20,DonWal22a}.
\begin{itemize}
\item 
As previously mentioned, our starting point is formulating the problem in the similarity coordinates $\tau=-\log(T-t)$ and $\rho=\frac{r}{T-t}$.
In these coordinates the blowup solution $u_*^T$ is just a constant function. 
Moreover, Eq.~\eqref{Eq:startingeq2} takes the
autonomous first-order form
\[ \Phi(\tau)=\mathbf L\Phi(\tau)+\mathbf N(\Phi(\tau)) \]
for perturbations $\Phi$ of the blowup, where $\mathbf L$ is a spatial
differential operator obtained from linearising the equation around the blowup solution, and $\mathbf N(\Phi(\tau))$ is the resulting nonlinearity. We show that $\mathbf L$ generates a
semigroup $\mathbf S$ on $\mathcal{H}:= H^1\times L^2(\B^d_1)$ and prove that $\Lf$ has precisely one
unstable eigenvalue $\lambda = 1$. This
eigenvalue does not correspond to a ``real'' instability of the
blowup solution $u^T$ but is just a consequence of the time
translation symmetry of the equation.
Furthermore, we establish that the associated spectral projection $\Pf$ is of rank $1$ and by standard semigroup theory we obtain the bounds
\begin{align*}
\|\Sf(\tau) (\I-\Pf)\|_{\mathcal{H}}\leq C_\varepsilon e^{\varepsilon\tau}
\end{align*}
for any $\varepsilon>0$
and 
\begin{align*}
\|\Sf(\tau) \Pf\|_{\mathcal{H}}\leq C e^{\tau}.
\end{align*}
 \item
Having done this groundwork, we begin with establishing Strichartz estimates for
$\mathbf S(\tau)(\I-\Pf)$.
For this, we first asymptotically construct the resolvent of $\Lf$.
One of the key ingredients for this is the fact that
the spectral equation $(\lambda-\Lf)\uf=\ff$ 
with $\uf=(u_1,u_2)$ and $\ff=(f_1,f_2)$ reduces to the second order ODE
\begin{equation}\label{eq:outline}
  \begin{split}
   (1-\rho^2)&u_1''(\rho)+\left(\frac{d-1}{\rho}-(2\lambda+d)\rho\right)u_1'(\rho)-\left(\lambda\left(\lambda+d-1\right)\right)u_1(\rho)
    \\
    &-du_1(\rho)=-F_\lambda(\rho)
    \end{split}
\end{equation}
with $F_\lambda(\rho)=f_2(\rho)+(\lambda+\frac d2)f_1(\rho)+\rho f_1'(\rho)$ and $\rho \in (0,1)$.
To simplify the analysis of the equation, it useful to get rid of the first
order derivative by appropriately transforming the
independent variable $u$. For the resulting equation,
we then construct fundamental systems near both poles separately.
To do so, we employ the diffeomorphism
\[
\varphi(\rho):=\frac{1}{2}\log\left(\frac{1+\rho}{1-\rho}\right)
\]
which, by means of a Liouville-Green transform, transforms the equation
into a Bessel equation.
Near $0$, we then construct two linearly independent solutions to
Eq.~\eqref{eq:outline}  which are in essence perturbed Bessel
functions and we control the error by Volterra iterations. 
However, near the endpoint $1$, we cannot use Hankel functions, as the resulting control over the perturbative would not be good enough for our purposes if we were to proceed in that fashion.
Hence, near 1, we have no choice but to construct a fundamental system directly. This is also done by means of Volterra iterations.
After gluing together the different solutions, we obtain a satisfactory representation of the
resolvent $(\lambda-\Lf)^{-1}$.

\item
To continue, we make use of the Laplace
representation
\[ \Sf(\tau)(\I-\Pf)\ff=\frac{1}{2\pi
    i}\lim_{N\to\infty}\int_{\epsilon-iN}^{\epsilon+iN}e^{\lambda\tau}(\lambda-\Lf)^{-1}(\I-\Pf)\ff\,d\lambda
\]
which, together with our asymptotic construction of the
resolvent of $\Lf$, allows us to push the contour of integration onto the imaginary axis.
We then proceed to 
prove Strichartz estimates by bounding the resulting oscillatory integrals.
More precisely, we prove
bounds for the difference of the linearised evolution to the free
evolution. This is much easier on a technical level and equivalent, given that Strichartz estimates for the free evolution follow easily from the standard Strichartz estimates in Cartesian coordinates. 
\end{itemize}

\subsection{Related results}
Critical dispersive equations have been the focus of extensive research over the last years and so, the singularity formation of such equations is also being intensively studied. As a result, there have been numerous intriguing works of which we can sadly only mention a handful. We begin with results on type II blowup (i.e., blowup solutions with finite energy norm). The construction of such solutions can be found in \cite{KriSchTat09}, \cite{KriSch14}, \cite{HilRap12} and \cite{Jen17}. Further, for the classification of type II blowup we refer to \cite{DuyKenMer11,DuyKenMer12a,DuyKenMer12b,DuyKenMer16}. Moreover, fascinating results concerning the stability of type II blowup, have been established in \cite{BurKri17} and \cite{Kri18}. Moving on to type I blowup, in \cite{Don17} and \cite{DonRao20}, the stability of the ODE blowup profile in the lightcone has already been established for the energy critical wave equation in three and five dimensions. Recently, the author together with Donninger also managed to show the blowup stability for $H^2\times H^1$ critical wave maps \cite{DonWal22a}. Stability for subcritical equations was shown by the impressive methods devised by Merle and Zaag in their joint works \cite{MerZaa03,MerZaa05,MerZaa07,MerZaa08,MerZaa12a,MerZaa12b,MerZaa15} and further investigated by Alexakis and Shao \cite{AleSha17} and Azaiez \cite{Aza15}. For an excellent numerical study of the blowup profile, we refer to the study  by  Bizo\'{n}, Chmaj, and Tabor \cite{BizChmTab04}. Asymptotic stability in the supercritical case was shown by Donninger and Schörkhuber, see \cite{DonSch16,DonSch17}. Recently, Glogić and Schörkhuber  also studied the cubic and, together with Csobo, the quadratic wave equation \cite{GloScho21,CsoGloSch21}. They show the stability of a blowup solution which is not independent of the spatial variables. We would also like to mention related developments in a more general coordinate system, called hyperboloidal similarity coordinates. These coordinates allow one to study blowup stability problems on much larger domains than just the backwards lightcone \cite{BieDonSch21,DonOst21}.  Furthermore, as our work pertains to the body of work on Strichartz estimates for wave equations, we also want to mention some of the more recent results on these, e.g., \cite{DAnFan08,MetTar07,DonGlo19}.

\section{Similarity coordinates and Semigroup Theory}
As previously stated, we let $d\geq 3$ be a fixed natural number. Then, the coordinates in which we study the evolution of the Cauchy problem \eqref{Eq:startingeq2} are the similarity coordinates given by 
\begin{align*}
\rho=\frac{r}{T-t},\qquad\tau=-\log(T-t)+\log(T).
\end{align*}
Upon setting $\psi(\tau,\rho)=(Te^{-\tau})^{\frac{d-2}{2}}u(T-Te^{-\tau},Te^{-\tau}\rho)$, Eq. \eqref{Eq:startingeq2} transforms into the equation
\begin{equation}\label{Eq: sim coordinates}
\begin{split}
\bigg[\partial_\tau^2 &+(d-1)\partial_\tau+2\rho\partial_\tau\partial_\rho -(1-\rho^2)\partial_\rho^2 -\frac{d-1}{\rho}\partial_\rho+d\rho\partial_\rho+\frac{d(d-2)}{4}\bigg]\psi(\tau,\rho)
\\
&=\psi(\tau,\rho)|\psi(\tau,\rho)|^{\frac{4}{d-2}}.
\end{split}
\end{equation}
Let now
\begin{align*}
\psi_1(\tau,\rho)&=\psi(\tau,\rho)
\\
\psi_2(\tau,\rho)&=\partial_\tau\psi(\tau,\rho)+\rho\partial_\rho\psi(\tau,\rho)+\frac{d-2}{2}\psi(\tau,\rho).
\end{align*}
Then Eq. \eqref{Eq: sim coordinates} turns into the system
\begin{equation} \label{Eq:intermediate system2}
\begin{split}
\partial_\tau \psi_1(\tau,\rho)&=-\rho\partial_\rho\psi_1(\tau,\rho)-\frac{d-2}{2}\psi_1(\tau,\rho)+\psi_2(\tau,\rho)
\\
\partial_\tau \psi_2(\tau,\rho)&=\partial_\rho^2\psi_1(\tau,\rho)+\frac{d-1}{\rho}\partial_\rho \psi_1(\tau,\rho)-\rho\partial_\rho \psi_2(\tau,\rho)-\frac{d}{2}\psi_2(\tau,\rho)
\\
&\quad+\psi_1(\tau,\rho)|\psi_1(\tau,\rho)|^{\frac{4}{d-2}},
\end{split}
\end{equation}
whereas the the blowup solution $u^T$ transforms into the constant function 
\[\begin{pmatrix}
c_d\\
\frac{d-2}{2}c_d
\end{pmatrix}.
\]
Next, we define $\H:=\{\uf \in H^1 \times L^2(\B^d_1):\uf \text{ radial}\} $ and denote by  $\|.\|_{\mathcal{H}}$ the standard radial $H^1\times L^2(\B^d_1)$ norm.
Motivated by the above system, we define the operator \\$\widetilde{\Lf}:D(\widetilde{\Lf})\subset  \H \to \H$, corresponding to the linear part in Eq.~\eqref{Eq:intermediate system2}, as
\begin{align*}
\widetilde{\Lf} \uf(\rho)=\begin{pmatrix}
-\rho u_1'(\rho)-\frac{d-2}{2}u_1(\rho)+u_2(\rho)\\
u_1''(\rho)+\frac{d-1}{\rho}u_1'(\rho)-\rho u_2'(\rho)-\frac{d}{2}u_2(\rho)
\end{pmatrix},
\end{align*}
where $D(\widetilde{\Lf}):=\{\uf\in C^2\times C^1(\overline{\B}_1^d):\uf \text{ radial} \}$.

Further, we define an inner product on $D(\widetilde{\Lf})$, as
\begin{align*}
\left(\uf|\vf\right)_{E}:=\int_0^1 u_1'(\rho)\bar{v}_2'(\rho)\rho^{d-1} d \rho+ \int_0^1 u_2(\rho)\bar{v}_2(\rho)\rho^{d-1} d \rho+ u_1(1)\bar{v}_1(1)
\end{align*}
for any $\uf,\vf \in D(\widetilde{\Lf})$ and we denote the associated norm by $\|.\|_{E}$.
\begin{lem}
The norms $\|.\|_{E}$ and $\|.\|_{\mathcal{H}}$ are equivalent on $D(\widetilde{\Lf})$, hence, also on $\H$.
\end{lem}
\begin{proof}
This follows in the same fashion as Lemma 2.2 in \cite{DonRao20}.
\end{proof}
This norm has the nice feature that we can fairly quickly establish a dispersive estimate on $\widetilde{\Lf}$.
\begin{lem}
We have that
\begin{align*}
\Re\left(\widetilde{\Lf} \uf|\uf\right)_{E}\leq 0
\end{align*}
for all $\uf\in C^2\times C^1(\overline{\B^d_1})$.
\end{lem}
\begin{proof}
An integration by parts shows that
\begin{align*}
-\Re \left(\int_0^1 u_1''(\rho) \overline{u}_1'(\rho)\rho^d d \rho
\right)=-\frac{|u_1'(1)|^2}{2}+\frac{d}{2}\int_0^1|u_1'(\rho)|^2 \rho^{d-1} d\rho
\end{align*}
and so
\begin{align*}
\Re\left(\int_0^1[\widetilde{\Lf} \uf]_1'(\rho) \bar{u}_1'(\rho) \rho^{d-1}d \rho\right)=&\Re\bigg(\frac{|u_1'(1)|^2}{2}
+\int_0^1 u_2'(\rho)\bar{u}_1'(\rho)\rho^{d-1}d \rho \bigg).
\end{align*}
Analogously, we see that
\begin{align*}
\Re\left(\int_0^1[\widetilde{\Lf} \uf]_2(\rho) \bar{u}_2(\rho) \rho^{d-1}d \rho\right)=&\Re\left(\int_0^1u_1''(\rho) \bar{u}_2(\rho)\rho^{d-1}+(d-1)u_1'(\rho) \bar{u}_2(\rho)\rho^{d-2}d \rho\right)
\\
&-\frac{1}{2}|u_2(1)|^2
\\
=&\Re\left(-\int_0^1u_1'(\rho) \bar{u}_2'(\rho)\rho^{d-1}+u_1'(1)\bar{u}_2(1)\right)
-\frac{1}{2}|u_2(1)|^2.
\end{align*}
Furthermore,
\begin{align*}
[\widetilde{\Lf} \uf]_1(1) \bar{u}_1(1)=-u_1'(1)\bar{u}_1(1)-\frac{d-2}{2}|u_1(1)|^2+u_2(1) \bar{u}_1(1)
\end{align*}
and thus, as $\frac{d-2}{2}\geq \frac{1}{2}$, using the elementary inequality
$$
\Re(-a\bar{b}+a\bar{c}+b\bar c) \leq\frac{1}{2}\left(|a|^2+|b|^2+|c|^2\right)
$$
with $a=-u_1'(1)$, $b=u_2(1)$, and $c=u_1(1)$,
implies that 
\begin{align*}
\Re(\widetilde{\Lf}\uf|\uf)_E\leq 0,
\end{align*}
for all $\uf \in D(\widetilde{\Lf})$.
\end{proof}
Next, we will show that the range $(1-\widetilde{\Lf})$ lies dense in $\mathcal{H}$.
To this end, note that
\begin{align*}
\left(\lambda-\widetilde{\Lf}\right)\uf=\ff
\end{align*}
implies 
\begin{align*}
\lambda u_1(\rho)+\rho u_1'(\rho)+\frac{d-2}{2}u_1(\rho)-u_2(\rho)=f_1(\rho)\\
\lambda u_2(\rho)- u_1''(\rho)-\frac{d-1}{\rho}u_1'(\rho)+\rho u'_2(\rho)+\frac{d}{2}u_2(\rho)=f_2(\rho).
\end{align*}
The first of these equations now implies
$$
u_2=\lambda u_1+\rho u_1'(\rho)+\frac{d-2}{2}u_1(\rho)-f_1(\rho)
$$
which we can plug into the latter to obtain a linear second order differential equation given by
\begin{align}\label{free generalised spectral eq}
(1-\rho^2)u_1''(\rho)+\left(\frac{d-1}{\rho}-(2\lambda+d)\rho\right)u_1'(\rho)-\left(\lambda\left(\lambda+d-1\right)+d\frac{d-2}{4}\right)u_1(\rho)=-F_\lambda(\rho)
\end{align}
with 
$F_\lambda(\rho)=f_2(\rho)+(\frac{d}{2}+\lambda) f_1(\rho)+ \rho f_1'(\rho)$.
\begin{lem}
The range of the operator $1-\widetilde{\Lf}$ lies dense in $\mathcal{H}$.
\end{lem}
\begin{proof}
For $\lambda=1$ the homogeneous version of Eq. \eqref{free generalised spectral eq} reads as
\begin{align}\label{lambda=1}
(1-\rho^2)u_1''(\rho)+\left(\frac{d-1}{\rho}-(d+2)\rho\right)u_1'(\rho)-
\frac{d}{4}(d+2) u_1(\rho)=0
\end{align}
and a fundamental system of solutions is given by
\begin{align*}
u_0(\rho)&=\frac{1}{(1+\sqrt{1-\rho^2})^{\frac{d}{2}-1}\sqrt{1-\rho^2}}\\
u_1(\rho)&=\frac{(1-\sqrt{1-\rho^2})^{\frac{d}{2}-1}-(1+\sqrt{1-\rho^2})^{\frac{d}{2}-1}}{\rho^{d-2}\sqrt{1-\rho^2}}.
\end{align*}
Moreover, their Wronskian is given by
\begin{align*}
W(u_0,u_1)(\rho)=\frac{d-2}{\rho^{d-1}(1-\rho^2)^\frac{3}{2}}.
\end{align*}
For $F_1$ in $C^\infty(\overline{\B^d_1}),$ we now use Duhamel's formula to see that 
an explicit solution of the equation
\begin{align*}
(1-\rho^2)u_1''(\rho)+\left(\frac{d-1}{\rho}-(d+2)\rho\right)u_1'(\rho)-
\frac{d}{4}(d+2) u_1(\rho)=-F_1(\rho)
\end{align*}
is given by
\begin{align*}
u(\rho)=-u_{0}(\rho)\int_\rho^1 \frac{u_{1}(s)F_1(s)}{(1-s^2)W(u_{0},u_{1})(s)} ds-u_{1}(\rho)\int_0^\rho \frac{u_{1}(s)F_1(s)}{(1-s^2)W(u_{0},u_{1})(s)} ds.
\end{align*}
From standard ODE theory, it follows that $u\in C^\infty((0,1))$ and therefore only the behavior at the end points needs to be checked. For this, we remark that the pole of $u_1$ at $\rho=0$ is of order $d-2$, and so,
it follows that $u$ is continuous on $[0,1)$ and, after a quick inspection, one concludes that $u$ is even continuous on $[0,1]$.
Next,
\begin{align}\label{Eq:first derivative}
u'(\rho)=-&u_{0}'(\rho)\int_\rho^1 \frac{s^{d-1}\sqrt{1-s^2} u_{1}(s)F_1(s)}{d-2} ds-u_{1}'(\rho)\int_0^\rho \frac{s^{d-1} \sqrt{1-s^2}u_{0}(s)F_1(s)}{d-2} ds
\end{align}
and using de l'Hopitals rule one verifies that $u'$ is continuous up to $\rho=1$. Likewise, it is straightforward to check that
\begin{align*}
\lim_{\rho\to 0}u_{1}'(\rho)\int_0^\rho \frac{s^{d-1} \sqrt{1-s^2}u_{0}(s)F_1(s)}{d-2} ds=0.
\end{align*}
Moreover, we observe that 
$u_0$ is a smooth even function on $(-1,1)$ and so, $ u_0'(0)=0$.
Finally, the second derivative of $u$ is given by
\begin{align*}
u''(\rho)&=-u_{0}''(\rho)\int_\rho^1 \frac{s^{d-1}\sqrt{1-s^2} u_{1}(s)F_1(s)}{d-2} ds-u_{1}''(\rho)\int_0^\rho \frac{s^{d-1} \sqrt{1-s^2}u_{0}(s)F_1(s)}{d-2} ds
\\
&\quad-\frac{F_1(\rho)}{1-\rho^2}
\end{align*}
A direct calculation shows that $u_1''(\rho)$ is continuous up to $\rho=1$, while it has a pole of order $d$ at $0$. Once more using de l'Hopitals rule therefore shows that 
\[u_{1}''(\rho)\int_0^\rho \frac{s^{d-1} \sqrt{1-s^2}u_{0}(s)F_1(s)}{d-2} ds \in C([0,1]).
\]
Consequently, we see that $u\in C^2([0,1))$.
Lastly, 
we note we can rewrite $u_0$ as
\begin{align*}
u_0(\rho)=\frac{\left(\sqrt{1+\rho}-\sqrt{1-\rho}\right)^{d-2}}{2^{\frac{d}{2}-1}\rho^{d-2}\sqrt{1-\rho^2}}=\frac{(1+\rho)^{\frac{d-2}{2}}}{2^{\frac{d}{2}-1}\rho^{d-2}\sqrt{1-\rho^2}} + f(\rho)+\sqrt{1-\rho}g(\rho)
\end{align*}
for some functions $f,g$ which are smooth on $(0,1]$. 
Hence, 
$u_0''$ is of the form
\begin{align*}
u_0''(\rho)=\frac{3(1+\rho)^{\frac{d-2}{2}}}{ 2^{\frac{d}{2}-1}\rho^{\frac{d}{2}-1}(1-\rho^2)^{\frac{5}{2}}}+(1-\rho)^{-\frac{3}{2}} h(\rho)
\end{align*}
with $h \in C(0,1])$.
Further, 
\[
\lim_{\rho\to 1}u_{1}(\rho)=2-d
\]
and so,
\[
\int_\rho^1 \frac{s^{d-1}\sqrt{1-s^2} u_{1}(s)F_1(s)}{d-2}\sim -F_1(1)\int_\rho^1 \sqrt{1-s^2} ds \sim-\frac{1}{3}(1-\rho^2)^{\frac{3}{2}}F_1(1)
\]
as $\rho \to 1$.
Consequently, the two singularities at $\rho=1$ cancel out exactly and we obtain that $u\in C^2([0,1])$, with $u'(0)=0$.
\end{proof}
Combining the last two Lemmas now yields 
\begin{lem}
The operator $\widetilde{\Lf}$ is closable and its closure, denoted by $\Lf_0$, generates a strongly continuous semigroup $\Sf_0$ on $ \mathcal{H}$, that satisfies
\begin{align*}
\|\Sf_0(\tau)\ff\|_{\mathcal{H}}\lesssim \|\ff\|_{\mathcal{H}},
\end{align*}
for all $\tau\geq 0$ and all $\ff\in \mathcal{H}$.
\end{lem}
\begin{proof}
This follows immediately from the last two lemmas and the Lumer Phillips theorem.
\end{proof}
\subsection{Strichartz estimates for the free equation}
As the existence of the ``free'' semigroup $\Sf_0$ has been established, the next step is to prove Strichartz estimates for this semigroup. To do so, we will from now on always assume that $T$ is confined to $\left[\frac{1}{2},\frac{3}{2}\right]$. This restriction of $T$ leads to no loss of generality as we are only interested in $T$ close to $1$ anyway. First, we will need the following technical Lemma.
\begin{lem}\label{lem:extension}
There exists a family of extension operators $\Ef_T:H^1\times L^2(\B^d_T) \to H^1\times L^2(\R^d)$ such that
\begin{align*}
\|\Ef_T\ff\|_{H^1\times L^2(\R^d)}\lesssim \|\ff\|_{H^1\times L^2(\B^d_T)}
\end{align*}
for all $T\in \left[\frac{1}{2},\frac{3}{2}\right]$ and all $\ff \in H^1\times L^2(\B^d_T)$.

\end{lem}
\begin{proof}
Let $\Ef_1$ be a bounded extension operator  from $H^1\times L^2(\B^d_1)$ to $ H^1\times L^2(\R^d)$.  For $\ff \in H^1\times L^2(\B^d_T)$ we define an extension by first mapping $\ff$ to $H^1 \times L^2(\B^d_1)$ via the scaling $\ff\mapsto \ff(T.)$, then extending via $\Ef_1$, and finally undoing the scaling. Since $T\in \left[\frac{1}{2},\frac{3}{2}\right]$, the resulting family of extension operators satisfies the desired estimate.
\end{proof}
\begin{lem}\label{lem: free strichartz}
Let $p\in[2,\infty]$ and $q \in [\frac{2d}{d-2},\frac{2d}{d-3}]$ be such that $\frac{1}{p}+\frac{d}{q}=\frac{d}{2}-1$. Then the estimate
\begin{align*}
\|[\Sf_0(\tau)\ff]_1\|_{L^{p}_\tau(\R_+)L^{q}(\B^d_1)}\lesssim\|\ff\|_{\mathcal{H}},
\end{align*}
holds for all $\ff\in C^2\times C^1(\overline{\B^d_1})$.
Additionally, the inhomogeneous estimate
\begin{align*}
\left\|\int_0^\tau\left[\Sf_0(\tau-\sigma)\hfh(\sigma,.)\right]_1 d\sigma\right\|_{L^p_\tau(I)L^q(\B^d_1)}\lesssim \|\hfh\|_{L^1(I)\mathcal{H}}
\end{align*}
holds for all $\hfh \in L^1(\R_+,\mathcal{H})\cap  C([0,\infty),\mathcal{H}) $ and all intervals $I
\subset [0,\infty)$ containing $0$.
\end{lem}
\begin{proof}
For $d=3$ this is Proposition 2.2 in \cite{Don17} and so we can assume $d\geq 4.$
Let $T\in \left[\frac{1}{2},\frac{3}{2}\right]$, $\ff \in C^2 \times C^1(\overline{\B^d_1})$, and define $\Af_T:H^1\times L^2(\B^d_T)\to \mathcal{H}$ by
 \begin{align*}
\Af_T \ff =( Tf_1(T.),T^2 f_2(T.)).
\end{align*}
In view of the coordinate transformations done in section 2, the evolution $\Sf_0(.)\ff$
is given by the solution $u\in C^2(\R_+\times \R)$ restricted to the
lightcone $\Gamma^T$ of the equation
\begin{align*}
\begin{cases}
\left(\partial_t^2-\partial_r^2-\frac{d-1}{r}\partial_r
\right)u(t,r)=0\\
(u(0,.),\partial_0 u(0,.))=\Ef_T\Af_T^{-1}\ff,
\end{cases}
\end{align*}
where $\Ef_T$ is the Sobolev
extension from Lemma \ref{lem:extension}.
Therefore, $$
\left[\Sf_0(\tau)\ff\right]_1(\rho)=(Te^{-\tau})^{\frac{d-2}{2}}u(T-Te^{-\tau},Te^{-\tau}\rho).
$$
Now, let $(p,q)=(\frac{1}{2},\frac{2d}{d-3})$ and calculate
\begin{align*}
\|\left[\Sf_0(\tau)\ff\right]_1\|_{L^{\frac{2d}{d-3}}(\B^d_1)}&\lesssim
\|e^{\frac{2-d}{2}\tau}u(T-Te^{-\tau},e^{-\tau}.)\|_{L^{\frac{2d}{d-3}}(\B^d_1)}\\
&= e^{-\frac{1}{2}\tau}\|u(T-Te^{-\tau},.)\|_{L^{\frac{2d}{d-3}}(\B^d_1)}
\\
&\leq e^{-\frac{1}{2}\tau}\|u(T-Te^{-\tau},.)\|_{L^{\frac{2d}{d-3}}(\R^d)}.
\end{align*}
So,
\begin{align*}
\|\left[\Sf_0(\tau)\ff\right]_1\|_{L^2_\tau(\R_+)L^{\frac{2d}{d-3}}(\B^d_1)}
&\lesssim \left (\int_0^\infty \left(e^{-\frac{1}{2}\tau}\|u(T-Te^{-\tau},.)\|_{L^{\frac{2d}{d-3}}(\R^d)}\right)^2 d \tau\right)^{\frac{1}{2}}
\\
&\lesssim\|u\|_{L^2(\R_+)L^{\frac{2d}{d-3}}(\R^d)}
\lesssim\left\|\Ef_T\Af_T^{-1}\ff\right\|_{H^1\times L^2(\R^d)}
\\
&\lesssim\left\|\Af_T^{-1}\ff\right\|_{H^1\times L^2(\B^d_T)} \lesssim \|\ff\|_{\mathcal{H}}
\end{align*}
by the classical Strichartz estimates (see \cite{Tao06}, p.78, Theorem 2.6).
Next, the estimate 
\begin{align*}
\left\|\left[\Sf_0(\tau)\ff\right]_1\right\|_{L^\infty_\tau(\R_+)L^{\frac{2d}{d-2}}(\B^d_1)}\lesssim\|\ff\|_{\mathcal{H}}
\end{align*}
follows from the Sobolev embedding $ H^1(\B^d_1) \hookrightarrow L^{\frac{2d}{d-2}}(\B^d_1) $.
To obtain the estimate for general admissible pairs $(p,q),$ we interpolate with $\theta=\frac{2d}{q}+3-d$ to obtain
\begin{align*}
\|\left[\Sf_0(.)\ff\right]_1\|_{L^q(\B^d_1)}\leq \|\left[\Sf_0(.)\ff\right]_1\|_{L^{\frac{2d}{d-2}}(\B^d_1)}^{\theta}\|\left[\Sf_0(.)\ff\right]_1\|_{L^{\frac{2d}{d-3}}(\B^d_1)}^{1-\theta}.
\end{align*}
Hence,
\begin{align*}
\|\left[\Sf_0(.)\ff\right]_1\|_{L^p(\R_+)L^q(\B^d_1)}\leq& \|\left[\Sf_0(.)\ff\right]_1\|_{L^{\infty}(\R_+)L^{\frac{2d}{d-2}}(\B^d_1)}^{\theta}\|\left[\Sf_0(.)\ff\right]_1\|_{L^{p(1-\theta)}(\R_+)L^{\frac{2d}{d-3}}(\B^d_1)}^{1-\theta}
\\
=&\|\left[\Sf_0(.)\ff\right]_1\|_{L^{\infty}(\R_+)L^{\frac{2d}{d-2}}(\B^d_1)}^{\theta}\|\left[\Sf_0(.)\ff\right]_1\|_{L^{2}(\R_+)L^{\frac{2d}{d-3}}(\B^d_1)}^{1-\theta},
\end{align*}
given that $p=\frac{2q}{(d-2)q-2d}$. To obtain the inhomogeneous estimate, we let $I=[0,\tau_0)\subset [0,\infty)$ and use Minkowski's inequality to estimate
\begin{align*}
&\left\|\int_0^\tau\left[\Sf_0(\tau-\sigma)\hfh(\sigma,.)\right]_1 d\sigma\right\|_{L^p_\tau(I) L^q(\B^d_1)} 
\\
=&\left\|\int_0^{\tau_0}1_{[0,\tau_0]}(\tau-\sigma)\left[\Sf_0(\tau-\sigma)\hfh(\sigma,.)\right]_1 d\sigma\right\|_{L^p_\tau(I)L^q(\B^d_1)}
\\
\leq &\int_0^{\tau_0}\left\|1_{[0,\tau_0]}(\tau-\sigma)\left[\Sf_0(\tau-\sigma)\hfh(\sigma,.)\right]_1 \right\|_{L^p_\tau(\R_+)L^q(\B^d_1)}d\sigma
\\
\leq &\int_0^{\tau_0}\left\|\left[\Sf_0(\tau)\hfh(\sigma,.)\right]_1 \right\|_{L^p_\tau(\R_+)L^q(\B^d_1)}d\sigma
\\
\lesssim &\int_0^{\tau_0}\left\|\hfh(\sigma,.)\right\|_{\mathcal{H}}d\sigma.
\end{align*}
\end{proof}
As we intend to study solutions close to $u^T$, our next step is to (formally) linearise the the nonlinearity around $u^T$. Motivated by this we
make the ansatz $\Psi=\Phi+(c_d,\frac{d-2}{2}c_d)$ and
 define the formal nonlinear operator $\Nf$ as
\begin{align*}
\Nf(\uf):=\begin{pmatrix}
0\\
N(u_1)
\end{pmatrix}
\end{align*}
with
\begin{align*}
N(u):=|c_d+u|^{\frac{4}{d-2}}(c_d+u)-c_d^{\frac{d+2}{d-2}}-\frac{2d+d^2}{4}u
\end{align*}
as well as the linear operator $\Lf': \mathcal{H} \to \mathcal{H}$
\begin{align*}
\Lf'\uf :=\begin{pmatrix}
0\\
\frac{2d+d^2}{4}u_1
\
\end{pmatrix}.
\end{align*}
Note that $\Lf'$ is a compact operator on $\mathcal{H}$ and so, as a consequence, we obtain that
$
\Lf:=\Lf_0+\Lf'$
generates a semigroup $\Sf$ on $\mathcal{H}$.
This allows us to formally rewrite our equation in Duhamel form as
\begin{align}\label{integraleq}
\Phi(\tau)=\Sf(\tau)\uf+\int_0^\tau \Sf(\tau-\sigma)\Nf(\Phi(\sigma)) d \sigma.
\end{align}
\subsection{Spectral analysis of the perturbed equation}
The next step now is to give a description of the spectrum of $\Lf$.
\begin{lem}\label{lem:spectrum L}
The spectrum of $\Lf$, denoted by $\sigma(\Lf)$, satisfies $$\sigma(\Lf)\subset\{z\in\C:\Re(z)\leq0\}\cup \{1\}.$$
Furthermore, the point spectrum of $\Lf$ is contained in the set $\{z\in\C:\Re(z)<0\}\cup \{1\}$.
\end{lem}
\begin{proof}
From the fact that $\Lf_0$ generates a $C_0$-semigroup we conclude that $$\sigma(\Lf_0)\subset \{z\in \C:\Re\lambda\leq 0\}.$$
Moreover, as $\Lf'$ is a compact perturbation, any spectral point $\lambda$ with $\Re(\lambda)>0$ has to be an eigenvalue of finite algebraic and geometric multiplicity.
By repeating previous calculations we obtain that the eigenvalue equation $(\lambda-\Lf)\uf=0$ is equivalent to the second order differential equation
\begin{align}\label{generalised spectral eq}
(1-\rho^2)u_1''(\rho)+\left(\frac{d-1}{\rho}-(2\lambda+d)\rho\right)u_1'(\rho)-\left(\lambda\left(\lambda+d-1\right)-d\right)u_1(\rho)=0
\end{align}
By  setting $z=\rho^2$ and $v(z)= u(\sqrt{z})$ Eq.~\eqref{generalised spectral eq} turns into 
\begin{align}\label{eq:hypergeo}
z(1-z)v''(z)&+\frac{d}{2} v'(z)-\left(\lambda+\frac{d+1}{2}\right)z v'(z) -\frac{1}{4}\left(\lambda\left(\lambda+d-1\right)-d\right)v_1(z)=0,
\end{align}
which, for $a=\frac{\lambda}{2}+\frac{d}{2},b=\frac{\lambda}{2}-\frac{1}{2}
$, and $c=\frac{d}{2}$ is of hypergeometric form, i.e.,
\begin{align*}
z(1-z)v''(z)+(c-(a+b+1)z)v'(z)-abv(z)=0.
\end{align*} 
Since the Frobenius indices at $\rho=0$ are given $\{0,\frac{2-d}{2}\}$, a fundamental system for Eq.~\eqref{eq:hypergeo} is given by the hypergeometric function
\begin{align*}
v_0(z)=&\, _2F_1(a,b;c;z)
\end{align*}
and a second solution which fails to be in $H^1(\B^d_1)$.
Therefore, any eigenfunction has to be a multiple of $u_0(\rho)=v_0(\rho^2)$.
We claim that for $\Re \lambda\geq 0$ $u_0(\rho)\in H^1(\B^d_1)$, if and only if $\lambda=1$.
That $1$ is indeed an eigenvalue follows immediately from
\begin{align*}
_1F_2\left(\frac{d+1}{2},0;\frac{d}{2};z\right)=1.
\end{align*}
For the other implication we first remark that
\begin{align*}
\partial_z&\, _2F_1(a,b;c;z)=\frac{ab}{c}\, _2F_1(a+1,b+1;c+1;z).
\end{align*}
Further, we infer that the asymptotic behavior of $_2F_1(a+1,b+1;c+1;z)$ as $z\to 1^-$,  is determined by $\Re(c-a-b-1)=-\frac{1}{2}-\lambda$. In particular,  as a consequence of Eq.~(10.12) on page 166 in \cite{Olv97} we have that for $\lambda\neq 1$ with $\Re\lambda\geq0$
\begin{align*}
\lim_{z \to 1 }\, _2F_1(a+1,b+1;c+1;z) (1-z)^{\frac{1}{2}+\lambda}=\frac{\Gamma(c+1)\Gamma(a+b-c+1)}{\Gamma(a+1)\Gamma(b+1)}.
\end{align*}
Consequently, for $\Re\lambda\geq 0$, $u_0'(\rho)$ can only be square integrable on $\B^d_1$ if $\lambda=1$.
\end{proof}
Next, we are going to compute the multiplicities of the eigenvalue $1$. 
\begin{lem}\label{lem:mulitplicities}
The algebraic and geometric multiplicities of the eigenvalue $1\subset \sigma(\Lf)$ both equal 1 and an eigenfunction is given by
\begin{align*}
\gf(\rho)=\begin{pmatrix}
2\\
d
\end{pmatrix}.
\end{align*}
\end{lem}
\begin{proof}
A direct calculation shows that $\gf$ is indeed an eigenfunction corresponding to the eigenvalue 1. Assume that there exists another eigenfunction $\widetilde{\gf}$. 
As before this implies that $\widetilde{g}_1$ solves the differential equation \eqref{generalised spectral eq}, which has a fundamental system given by $ g_1$ and 
\begin{align*}
h_1(\rho)&=\begin{cases} \int_{\frac{1}{2}}^\rho \frac{s^{1-d}}{(1-s^2)^{\frac{3}{2}}}ds \text{ for }\rho \geq \frac{1}{2}
\\ 
-\int_\rho^{\frac{1}{2}} \frac{s^{1-d}}{(1-s^2)^{\frac{3}{2}}}ds \text{ for }\rho < \frac{1}{2}.
\end{cases}
\end{align*}
As $h_1$ is not in $H^1(\B^d_1)$,
 $\widetilde{g}_1$ has to be a multiple of $g_1$ and since the second component is uniquely determined by the first, we infer that $\widetilde{\gf}=c\gf$ for some constant $c\in \C$. \\
Let now $\Pf$ be the spectral projection associated to the eigenvalue $1$. We split $H^1\times L^2(\B^d_1)$ into the two closed subspaces $\rg\Pf$ and $\ker\Pf$. Analogously, we decompose the operator $\Lf$ into $\Lf_{\rg\Pf}:\rg\Pf \to \rg \Pf$ and 
$\Lf_{\ker\Pf}:\ker\Pf \to \ker \Pf$. The inclusion $span\{ \gf \} \subset \rg\Pf$ is immediate. For the other one we remark that the operator $(\textbf{1}-\Lf_{\rg\Pf})$ is nilpotent as the only eigenvalue of $\Lf_{\rg\Pf}$ is $1$. Let $n\geq1$ be the minimal index such that $(\textbf{1}-\Lf_{\rg\Pf})^n\uf =\textbf{0}$ for all $\uf \in \rg\Pf \cap D(\Lf).$
If $n=1$ we are done. If not, there exists a nontrivial $\vf \in D(\Lf)$ with $(1-\Lf)\vf =\gf$, which implies 
\begin{align}
(1-\rho^2)v_1''(\rho)+\left(\frac{d-1}{\rho}-(d+2)\rho\right)v_1'(\rho)=-(2d+2).
\end{align}
Since the Wronskian of $g_1$ and $h_1$ is given by
\begin{align*}
W(g_1,h_1)(\rho)=2\rho^{1-d}(1-\rho^2)^{-\frac{3}{2}},
\end{align*}
the variations of constants formula suggests that $v_1$ has to be of the form
\begin{align*}
v_1(\rho)=&c_0+c_1 h_1(\rho)+(2d+2)h_1(\rho)\int_{\rho}^{1} s^{d-1}(1-s^2)^{\frac{1}{2}} ds
\\
&+(2d+2)\int_{0}^{\rho} s^{d-1}(1-s^2)^{\frac{1}{2}}
h_1(s) ds
\end{align*}
for $c_0,c_1\in \C$.
Note, that $h_1(\rho)\int_{\rho}^{1} s^{d-1}(1-s^2)^{\frac{1}{2}} ds \in L^2([\frac{1}{2},1])$.
Moreover,  $$s^{d-1}(1-s^2)^{\frac{1}{2}}
h_1(s) $$ is integrable on $[0,1]$ and so, given that $h_1$ is not in $L^2(\B^d_1)$, $c_1$ 
has to equal $0$. However, from the divergence of $h_1$ at $\rho=0$, we see that for $v_1$ to be an element of $H^1(\B^d_1)$ the integral 
\begin{align*}
\int_{0}^{1} s^{d-1}(1-s^2)^{\frac{1}{2}} ds
\end{align*}
would have to vanish, which, by the positivity of the integrand on $(0,1)$, cannot happen.
\end{proof}
Next, we obtain a useful bound on the resolvent of $\Lf$.
\begin{lem}\label{lem:resbound}
For every $\varepsilon>0$ there exist constants $R_\varepsilon>0$ and $C_\varepsilon>0$ such that the resolvent operator of $\Lf$ denoted by $\Rf_{\Lf}(\lambda)$ satisfies the estimate
\begin{align*}
\|\Rf_{\Lf}(\lambda)\|\leq C_\varepsilon
\end{align*}
for all $\lambda \in \C$ with $\Re\lambda\geq \varepsilon$ and $|\lambda|\geq R_\varepsilon$.
\end{lem}
\begin{proof}
This can be proven in the same way as Lemma 2.6 in \cite{DonRao20}.
\end{proof}
By employing the Gearhart-Prüss Theorem we obtain the following result as a direct consequence of Lemmas \ref{lem:mulitplicities} and \ref{lem:resbound}.
\begin{lem}
For any $\varepsilon>0$, there exists a constant $C_\varepsilon>0$ such that
\begin{align*}
\|\Sf(\tau)(\I-\Pf)\ff\|_{\mathcal{H}}\leq C_\varepsilon e^{\varepsilon\tau}\|(\I-\Pf)\ff\|_{\mathcal{H}}
\end{align*}
for all $\tau\geq 0$ and all $\ff\in \mathcal{H}$.
\end{lem}
\section{Construction of the Green function}
In this section we take a closer look at the generalized spectral equation
\begin{equation}\label{generaleq}
\begin{split}
(1-\rho^2)u_1''(\rho)+\left(\frac{d-1}{\rho}-(2\lambda+d)\rho\right)u_1'(\rho)&-\left(\lambda\left(\lambda+d-1\right)+\frac{d(d-2)}{4}\right)
\\
&+V(\rho)u_1(\rho)=-F_\lambda(\rho)
\end{split}
\end{equation}
for an arbitrary potential $V\in C^{\infty}([0,1])$.
To study Eq. \eqref{generaleq}, we will often make use of functions of symbol type, which we define as follows. Let $I\subset \R$, $\rho_0\in\overline{I}$, and $\alpha \in \R$. We say that a smooth function $f:I \to \R$ is of symbol type and write $f(\rho)=\O((\rho_0-\rho)^{\alpha})$ if
\begin{align*}
|\partial_\rho^n f(\rho)|\lesssim_n |\rho_0-\rho|^{\alpha-n},
\end{align*}
 for all $\rho \in I$ and all $n\in \mathbb{N}_0$.
A discussion of symbol calculus can for instance be found in \cite{Don17} and \cite{DonRao20}. We will also make use of the ``Japanese bracket'' $\langle x\rangle:=\sqrt{1+|x|^2}$ and lastly, whenever we are given a function of the form $f(\rho,\lambda)$, then $f'(\rho,\lambda)$ denotes $\partial_\rho f(\rho,\lambda)$.
\subsection{Construction of a fundamental system}
To rid ourselves of the first order term we set
$$v(\rho):=\rho^{\frac{d-1}{2}}(1-\rho^2)^{\frac{1}{4}+\frac{\lambda}{2}}u(\rho),$$
 which, for $F_\lambda=0$ turns Eq.~\eqref{generaleq} into
 \begin{align}\label{nofirstorder}
v''(\rho)+\frac{(6+4\lambda-4\lambda^2)\rho^2+4d(1-\rho^2)-d^2(1-\rho^2)-3}{4\rho^2(1-\rho^2)^2}v(\rho)=\frac{V(\rho)}{1-\rho^2}v(\rho).
 \end{align}
As we are only interested in values of $\lambda$ that are close to the imaginary axis,
we assume that $\lambda=\varepsilon+ i\omega$ with $ \varepsilon \in [0,\frac{1}{4}]\cup [1-\frac{1}{4},1]$ and $\omega \in \R$.
To analyse Eq.~\eqref{nofirstorder}, we will make use of the diffeomorphism $\varphi:(0,1)\to (0,\infty)$, which is given by
 $$
 \varphi(\rho)=\frac{1}{2}(\log(1+\rho)-\log(1-\rho)).
 $$
Note that 
 $$
 \varphi'(\rho):=\frac{1}{(1-\rho^2)}
 $$
 and that the associated Liouville-Green Potential $Q_{\varphi}$, which is defined as
 $$
 Q_\varphi(\rho)=-\frac{3}{4}\frac{\varphi''(\rho)^2}{\varphi'(\rho)^2}+\frac{1}{2}\frac{\varphi'''(\rho)}{\varphi'(\rho)^2},
 $$ is given by
 $$
 Q_{\varphi}(\rho )=\frac{1}{(1-\rho^2)^2}.
 $$
 Motivated by this, we rewrite Eq. \eqref{nofirstorder} as
  \begin{align}
&v''(\rho)-\left(\frac{1-4\lambda+4\lambda^2}{4(1-\rho^2)^2}-\frac{4d-d^2-3}{4\varphi(\rho)^2(1-\rho^2)^2}+Q_{\varphi}(\rho)\right)v(\rho)\nonumber
\\
=&\left(-\frac{4d-d^2-3}{4\rho^2(1-\rho^2)}+\frac{4d-d^2-3}{4\varphi(\rho)^2(1-\rho^2)^2}+\frac{V(\rho)}{1-\rho^2}\right)v(\rho). 
 \end{align}
We now perform a Liouville-Green transformation, i.e., we set  $w(\varphi(\rho)):= \varphi'(\rho)^{\frac{1}{2}}v(\rho)$ which transforms
\begin{align}\label{beforebessel}
v''(\rho)-\left(\frac{1-4\lambda+4\lambda^2}{4(1-\rho^2)}-\frac{4d-d^2-3}{4\varphi(\rho)^2(1-\rho^2)^2}+Q_{\varphi}(\rho)\right)v(\rho)=0
\end{align}
into the equation
\begin{equation}\label{bessel}
w''(\varphi(\rho))-\left(\frac{1}{2}-\lambda\right)^2w(\varphi(\rho))+\frac{4d-d^2-3}{4\varphi(\rho)^2}w(\varphi(\rho))=0.
\end{equation}
Eq.~\eqref{bessel} is a Bessel equation and a fundamental system for it is given by
\begin{align*}
&\sqrt{\varphi(\rho)}J_{\frac{d-2}{2}}(a(\lambda)\varphi(\rho))\\
&\sqrt{\varphi(\rho)}Y_{\frac{d-2}{2}}(a(\lambda)\varphi(\rho))
\end{align*}
where $J_\nu $ and $Y_\nu$ are the Bessel functions of the first and second kind respectively, and with $a(\lambda)=i(\frac{1}{2}- \lambda)$. From this, we infer that Eq.~\eqref{beforebessel} has the two linearly independent solutions
\begin{align*}
b_{1}(\rho,\lambda)&= \sqrt{(1-\rho^2)\varphi(\rho)}J_{\frac{d-2}{2}}(a(\lambda)\varphi(\rho))\\
b_{2}(\rho,\lambda)&=\sqrt{(1-\rho^2)\varphi(\rho)}Y_{\frac{d-2}{2}}(a(\lambda)\varphi(\rho)).
\end{align*}
Consider now the set
\begin{align*}
\{z\in \C: 0<|z|< r\}\cap\{ z \in \C: 0 \leq \arg z <\pi\} 
\end{align*} 
for $r>0$ fixed, where $\arg$ denotes the (principal branch of the) argument of $z$. A Taylor expansion shows that on this set, the Bessel function $J_\nu$ satisfies
\begin{align}\label{rep:BesselJ}
J_\nu(z)= z^\nu \left(\frac{1}{2^\nu\Gamma(\nu+1)}+\O(z^2)\right).
\end{align}
Further, when $\nu$ is an integer we have that
\begin{align}\label{rep:BesselYint}
Y_\nu(z)= -z^{-\nu}\left(2^\nu\frac{(\nu-1)!}{\pi }+\O(z)\right),
\end{align}
while for half integers
\begin{align}\label{rep:BesselYhalf}
Y_\nu(z)= -z^{-\nu}\left(\frac{2^\nu}{\sin(\nu \pi)\Gamma(-\nu+1)}+\O(z)\right).
\end{align}
In the half integer case we could of course explicitly calculate $J_\nu$ and $Y_\nu$ via their recurrence relations, this is however not needed.
\begin{lem} \label{Lem:hom near 1}
Let $\rho \in [\frac{1}{|a(\lambda)|},1)$. Then,
the equation 
 \begin{align}\label{Eq:hom near 1}
v''(\rho)+\frac{(6+4\lambda-4\lambda^2)\rho^2+4d(1-\rho^2)-d^2(1-\rho^2)-3}{4\rho^2(1-\rho^2)^2}v(\rho)=0
 \end{align}
 has a fundamental system given by
 \begin{align*}
 h_1(\rho,\lambda)&=\frac{(1+\rho)^{\frac{3}{4}-\frac{\lambda}{2}}(1-\rho)^{\frac{1}{4}+\frac{\lambda}{2}}}{\sqrt{a(\lambda)}}[1+\O(\rho^{-1}(1-\rho)\langle\omega\rangle^{-1})]
 \\
 h_2(\rho,\lambda)&=\frac{(1+\rho)^{\frac{1}{4}+\frac{\lambda}{2}}(1-\rho)^{\frac{3}{4}-\frac{\lambda}{2}}}{\sqrt{a(\lambda)}}[1+\O(\rho^{-1}(1-\rho)\langle\omega\rangle^{-1})].
 \end{align*}
 
\end{lem}
\begin{proof}
We start by rewriting Eq.~\eqref{Eq:hom near 1} as
 \begin{align*}
v''(\rho)+\frac{3+4\lambda-4\lambda^2}{4(1-\rho^2)^2}v(\rho)=\frac{d^2+3-4d}{4\rho^2(1-\rho^2)}v(\rho)
 \end{align*}
and compute that the equation 
 \begin{align*}
w''(\rho)+\frac{3+4\lambda-4\lambda^2}{4(1-\rho^2)^2}w(\rho)=0
 \end{align*}
 has a fundamental system of solutions given by 
 \begin{align*}
 w_1(\rho,\lambda)&=\frac{(1+\rho)^{\frac{3}{4}-\frac{\lambda}{2}}(1-\rho)^{\frac{1}{4}+\frac{\lambda}{2}}}{\sqrt{a(\lambda)}}
 \\
 w_2(\rho,\lambda)&=w_1(\rho,1-\lambda)=\frac{(1+\rho)^{\frac{1}{4}+\frac{\lambda}{2}}(1-\rho)^{\frac{3}{4}-\frac{\lambda}{2}}}{\sqrt{a(\lambda)}}.
 \end{align*}
Performing a Volterra iteration will enable us to construct a solution to Eq. \eqref{Eq:hom near 1} out of these two functions. For this, we compute that the Wronskian of $w_1$ and $w_2$ is given by
$$
W(w_1(.,\lambda),w_2(.,\lambda))=2i
$$
and so, from the variation of constants formula, we derive the fixed point equation
\begin{align*}
w(\rho,\lambda)=&w_1(\rho,\lambda)+\int_{\rho}^{\rho_1}\frac{w_1(\rho,\lambda)w_2(s,\lambda)\widetilde{V}}{2is^2(1-s^2)}w(s,\lambda) ds
\\
&-\int_{\rho}^{\rho_1}\frac{w_2(\rho,\lambda)w_1(s,\lambda)\widetilde{V}}{2i s^2(1-s^2)}w(s,\lambda) ds
\end{align*}
for 
$$
\widetilde{V}:=d^2+3-4d,
$$
$\rho \in [\frac{1}{|a(\lambda)|},1)$, and with $\rho_1 \in [\frac{1}{|a(\lambda)|},1] $ to be chosen.
As $w_1(.,\lambda)$ is nonvanishing on $[\frac{1}{|a(\lambda)|},1)$ we can divide the whole equation by it. By introducing the variable 
$\widetilde{w}:=\frac{w}{w_1}$, we then obtain 
\begin{align}\label{volterra near 1}
\widetilde{w}(\rho,\lambda)&=1+\int_{\rho}^{\rho_1}\frac{w_1(s,\lambda)w_2(s,\lambda)\widetilde{V}}{2is^2(1-s^2)}\widetilde{w}(s,\lambda) ds
\\
&\quad-\int_{\rho}^{\rho_1}\frac{w_2(\rho,\lambda)w_1(s,\lambda)^2\widetilde{V}}{w_1(\rho,\lambda) 2i s^2(1-s^2)}\widetilde{w}(s,\lambda) ds \nonumber
\\
&=1+\int_{\rho}^{\rho_1}\frac{\left[(1-s^2)-(1+s^2)\left(\frac{1-\rho}{1+\rho}\right)^{\frac{1}{2}-\lambda}\left(\frac{1-s}{1+s}\right)^{\frac{1}{2}+\lambda}\right]\widetilde{V}}{s^2(2\lambda-1)}\widetilde{w}(s,\lambda) ds
\\
:&=1+\int_\rho^{\rho_1} K(\rho,s,\lambda)\widetilde{w}(s,\lambda) ds.\nonumber
\end{align}
We see that
\begin{align*}
\int_{\frac{1}{|a(\lambda)|}}^{\rho_1} \sup_{\frac{1}{|a(\lambda)|}\leq \rho <s}\left|K(\rho,s,\lambda)\right| ds\lesssim \int_{\frac{1}{|a(\lambda)|}}^{\rho_1}\frac{1}{s^2|1-2\lambda|} ds \lesssim 1
\end{align*}
and so, we can set $\rho_1=1$ and
use Lemma B.1 in \cite{DonSchSof11} to obtain that there exists a
unique solution $\widetilde w$ to Eq.~\eqref{volterra near 1} of the form 
$
\widetilde w(\rho,\lambda)=1+ O(\rho^{-1}\langle\omega\rangle^{-1}).
$
By re-inserting this expression into Eq.~\eqref{volterra near 1}, we
derive the improved representation 
$$\widetilde w(\rho,\lambda)=1+O(\rho^{-1}(1-\rho)\langle\omega\rangle^{-1}).
$$
To be completely precise, the $O$-term also depends on
$\varepsilon=\Re\lambda$ but this dependence is of no relevance to
us and, hence, suppressed in our notation.
To show that this solution is of symbol type, one can employ the diffeomorphism $\varphi$ and argue as in the proof of Lemma 3.3 in \cite{Don17}. Subsequently, we infer the existence of a solution to Eq. \eqref{Eq:hom near 1} given by
$$
h_1(\rho,\lambda):=w_1(\rho,\lambda)[1+\O(\rho^{-1}(1-\rho)\langle\omega\rangle^{-1})].
$$
The second solution $h_2$, is obtained by defining $h_2(\rho,\lambda):=h_1(\rho,1-\lambda)$.
\end{proof}
To continue, pick $r>0$ and $\rho_0 \in [0,1)$ such that $h_1$ and $h_2$ are non-vanishing on both the interval $[\frac{r}{|a(\lambda)|}, 1)$ and $[\rho_0,1)$. We now set $\rho_\lambda:=\min \{\rho_0,\frac{r}{|a(\lambda)|}\}$,
and observe that $h_1$ and $h_2$ do not vanish on $[\rho_\lambda,1)$. 

\begin{lem}
Let $\rho\in [\rho_\lambda,1)$. Then Eq.~\eqref{nofirstorder} has a fundamental system of the form
\begin{align*}
\psi_3(\rho,\lambda)&= h_1(\rho,\lambda)[1+\O(\rho^{0}(1-\rho)\langle\omega\rangle^{-1})]
 \\
 \psi_4(\rho,\lambda)&=h_2(\rho,\lambda)[1+\O(\rho^{0}(1-\rho)\langle\omega\rangle^{-1})].
\end{align*} 

\end{lem}
\begin{proof}
This follows by doing a Volterra iteration based on $h_1$ and $h_2$ and using the $\lambda$ symmetry of Eq.~\eqref{nofirstorder} as in the proof of Lemma \ref{Lem:hom near 1}.
\end{proof}
Before we can construct a fundamental system for Eq.~\eqref{nofirstorder} near $0$, we restrict $\Re \lambda=\varepsilon$ to the interval $[0,\frac{1}{4}]$ and define $\widehat{\rho}_\lambda:=\min \{\frac{1}{2}(1+\rho_0),\frac{2r}{|a(\lambda)|}\}$.
\begin{lem}\label{fundi near 0}
Let $\Re\lambda \in [0,\frac{1}{4}]$ and $\rho \in (0,\widehat{\rho}_\lambda]$. Then, Eq. \eqref{nofirstorder} has a fundamental system of the form
\begin{align*}
\psi_1(\rho,\lambda)=&b_1(\rho,\lambda)\left[1+\O(\rho^2\langle\omega\rangle^{0})\right] 
\\
=&\sqrt{(1-\rho^2)\varphi(\rho)}J_{\frac{d-2}{2}}( a(\lambda)\varphi(\rho))[1+\O(\rho^2\langle\omega\rangle^0)]
\\
\psi_2(\rho,\lambda)=& b_2(\rho,\lambda)[1+\O(\rho^2\langle\omega\rangle^0)]+\O(\rho^{\frac{7-d}{2}}\langle\omega\rangle^{1-\frac{d}{2}})
\\
=&\sqrt{(1-\rho^2)\varphi(\rho)}Y_\frac{d-2}{2}( a(\lambda)\varphi(\rho))[1+\O(\rho^2\langle\omega\rangle^0)]+ \O(\rho^{\frac{7-d}{2}}\langle\omega\rangle^{1-\frac{d}{2}}).
\end{align*}
\end{lem}
\begin{proof}
Recall that $b_1$ and $b_2$ are two linearly independent solutions to the equation
\begin{align*}
v''(\rho)-\left(\frac{1-4\lambda+4\lambda^2}{4(1-\rho^2)}-\frac{4d-d^2-3}{4\varphi(\rho)^2(1-\rho^2)^2}+Q_{\varphi}(\rho)\right)v(\rho)=0.
\end{align*}
Let now $$\widetilde{V}(\rho)=\left(-\frac{4d-d^2-3}{4\rho^2(1-\rho^2)}+\frac{4d-d^2-3}{4\varphi(\rho)^2(1-\rho^2)^2}+\frac{V(\rho)}{1-\rho^2}\right)$$
and note that a Taylor expansion shows $\widetilde{V}\in C^\infty([0,\widehat{\rho}_\lambda])$.
Furthermore, we have $$W(b_1(.,\lambda),b_2(.,\lambda))=\frac{2}{\pi},$$
thus, in order to find solutions of Eq. \eqref{nofirstorder} on $(0,\widehat{\rho}_\lambda)$ we make the ansatz
\begin{align}\label{Ansatz1}
b(\rho,\lambda)= b_1(\rho,\lambda)&-\frac{\pi}{2}b_1(\rho,\lambda)\int_{0}^{\rho} b_2(s,\lambda)\widetilde{V}(s)b(s,\lambda)d s\\
&+\frac{\pi}{2}b_2(\rho,\lambda)\int_{0}^{\rho} b_1(s,\lambda)\widetilde{V}(s,\lambda)b(s,\lambda) d s. \nonumber
\end{align}
We now record some useful facts, the first being that, as $\varphi(\rho)\sim \rho$ as $\rho \to 0 $, the condition $\rho \in (0,\widehat{\rho}_\lambda]$ implies that $a(\lambda)\varphi(\rho)$ is contained in the set $\{z\in \C: \Im(z)>0, 0<|z|\leq R\}$ for a sufficiently large $R>0$. Thus, the identities \eqref{rep:BesselJ}, \eqref{rep:BesselYint}, and \eqref{rep:BesselYhalf} hold. Secondly, since we are a definite distance away from the endpoint $\rho=1$ one can directly check the identity 
$$
\varphi(\rho)=\O(\rho)
$$
for $\rho \in (0,\widehat{\rho}_\lambda)$.
Lastly, as $\Re a(\lambda)\neq 0$ for all $\lambda\in \C$ with $\Re\lambda\in [0,\frac{1}{4}]$, $b_1$ is nonvanishing on $(0,\widehat{\rho}_\lambda]$, as all zeros of $J_\nu$ are real provided that $\nu>-1$ (see \cite{Olv97}, p. 244 Theorem 6.2). This allows us to divide the whole equation~ \eqref{Ansatz1} by $b_1$, which, upon introducing the help variable $\widetilde{b}=\frac{b}{b_1}$,
yields the integral equation
\begin{equation}\label{int0}
\widetilde{b}(\rho,\lambda)=1+\int_{0}^{\rho}K(\rho,s,\lambda)\widetilde{b}(s,\lambda) ds,
\end{equation}
with 
$$
K(\rho,s,\lambda)=\frac{\pi\widetilde{V}(s)}{2}\left( \frac{b_2(\rho,\lambda)}{b_1(\rho,\lambda)}b_1(s,\lambda)^2-b_2(s,\lambda)b_1(s,\lambda)\right).
$$
From \eqref{rep:BesselJ}, \eqref{rep:BesselYint}, and \eqref{rep:BesselYhalf}, we infer that 
$$
b_2(s,\lambda)b_1(s,\lambda)=\O(s \langle\omega\rangle^0) 
$$
and 
$$
\frac{b_2(\rho,\lambda)}{b_1(\rho,\lambda)}b_1(s,\lambda)^2=\O(\rho^{2-d}s^{d-1}\langle\omega\rangle^0)
$$
for $0<s\leq \rho<\widehat{\rho}_\lambda$.
Hence,
$$
\int_0^{\widehat{\rho}_\lambda} \sup_{\rho\in [s,\widehat\rho_\lambda]}|K(\rho,s,\lambda)|ds \lesssim \langle\omega\rangle^{-2}.
$$  Consequently, there exists a solution $\widetilde{b}$ to  Eq. \eqref{int0} that satisfies
\begin{align*}
\widetilde{b}(\rho,\lambda)=1+O(\rho^2\langle\omega\rangle^0).
\end{align*}
However, as all involved terms are of symbol form, we immediately conclude that
\begin{align*}
\widetilde{b}(\rho,\lambda)=1+\O(\rho^2\langle\omega\rangle^0).
\end{align*}
This in turn implies the existence of a solution of Eq. \eqref{nofirstorder} which takes the form
\begin{align*}
\psi_1(\rho,\lambda)&= \sqrt{(1-\rho^2)\varphi(\rho)}J_{\frac{d-2}{2}}( a(\lambda)\varphi(\rho))[1+\O(\rho^2\langle\omega\rangle^0)].
\end{align*}
To construct the second solution stated in the lemma, we pick a $\rho_1 \in (0,1]$ such that $\psi_1$ does not vanish for $\rho\leq\min\{\rho_1,\widehat{\rho}_\lambda\}=:\widetilde{\rho}_\lambda$ for any $0\leq \Re\lambda\leq \frac{1}{4}.$ Next, as
$\widetilde{b}_1(\rho,\lambda):=b_1(\rho,\lambda)\int_{\rho}^{\widetilde{\rho}_\lambda} b_1(s,\lambda)^{-2} ds$  also solves Eq.~\eqref{beforebessel}, there exist constants $c_1(\lambda),c_2(\lambda)\in \C$ such that
\begin{align*}
b_2(\rho,\lambda)=c_1(\lambda) b_1(\rho,\lambda)+ c_2(\lambda)\widetilde{b}_1(\rho,\lambda).
\end{align*}
Moreover, we have the explicit formula
\begin{align*}
c_1(\lambda)&=\frac{W(b_2(.,\lambda),\widetilde{b}_1(.,\lambda))}{W(b_1(.,\lambda),\widetilde{b}_1(.,\lambda))}\\
c_2(\lambda)&=-\frac{W(b_2(.,\lambda),b_1(.,\lambda))}{W(b_1(.,\lambda),\widetilde{b}_1(.,\lambda))}.
\end{align*}
Using that
$W(b_2(.,\lambda),b_1(.,\lambda))=-\frac{2}{\pi}$ and $ W(b_1(.,\lambda),\widetilde{b}_1(.,\lambda))=-1$, we infer that
$c_2=-\frac{2}{\pi}$ and $c_1(\lambda)=-W(b_2(.,\lambda),\widetilde{b_1}(.,\lambda))$.
Evaluating $W(b_2(.,\lambda),\widetilde{b}_1(.,\lambda))$ at $\widetilde{\rho}_\lambda$ yields
\[
W(b_2(.,\lambda),\widetilde{b_1}(.,\lambda))=-b_2(\widetilde{\rho}_\lambda,\lambda)b_1(\widetilde{\rho}_\lambda,\lambda)^{-1}=\O(\langle\omega\rangle^{0}).
\]
Keeping these facts in mind, we now turn our attention to $\psi_2$ and remark that a second solution of Eq.~\eqref{nofirstorder} is given by
$\widetilde{\psi}_1(\rho,\lambda)=\psi_1(\rho,\lambda)\int_{\rho}^{\widetilde{\rho}_\lambda} \psi_1(s,\lambda)^{-2} ds$.
Considering this, we calculate
\begin{align*}
\psi_2(\rho,\lambda):&=c_1(\lambda)\psi_1(\rho,\lambda)+ c_2\psi_1(\rho,\lambda)\int_{\rho}^{\widetilde{\rho}_\lambda} \psi_1(s,\lambda)^{-2}ds 
\\
&=c_1(\lambda)\psi_1(\rho,\lambda)+
   c_2\psi_1(\rho,\lambda)\int_{\rho}^{\widetilde{\rho}_\lambda}
   b_1(s,\lambda)^{-2}ds \\
&\quad +c_2\psi_1(\rho,\lambda)\int_{\rho}^{\widetilde{\rho}_\lambda}\left [
   \psi_1(s,\lambda)^{-2}-b_1(s,\lambda)^{-2}\right ] ds
\\
&=b_2(\rho,\lambda)[1+\O(\rho^2\langle\omega\rangle^0)]
+c_2\psi_1(\rho,\lambda)\int_{\rho}^{\widetilde{\rho}_\lambda} \frac{\O(s^2\langle\omega\rangle^0)}{b_1(s,\lambda)^2[1+\O(s^2\langle\omega \rangle^0)]^2} ds.
\end{align*}
From $b_1(\rho,\lambda)^{-2}=\O(\rho^{1-d}\langle\omega\rangle^{2-d})$,
we obtain
 \[
\int_{\rho}^{\widetilde{\rho}_\lambda} \frac{\O(s^2\langle\omega\rangle^0)}{b_1(s,\lambda)^2[1+\O(s^2\langle\omega \rangle^0)]^2} ds=\O(\rho^0\langle\omega\rangle^{-2})+\O(\rho^{4-d}\langle\omega\rangle^{2-d})=\O(\rho^{4-d}\langle\omega\rangle^{2-d})
\]
and so, given that
$\psi_1(\rho,\lambda)=\O(\rho^{\frac{d-1}{2}}\langle\omega\rangle^{\frac{d-2}{2}})$, we see that $\psi_2$ is of the claimed form for $\rho\in (0,\widetilde{\rho}_\lambda)$.
However, since for $|\lambda|$ large enough we have that
$\widetilde{\rho}_\lambda=\widehat{\rho}_\lambda$, we can safely
assume that $\widetilde{\rho}_\lambda=\widehat{\rho}_\lambda$.
\end{proof}
Our next step will be to patch our solutions together.
\begin{lem}\label{connectioncoef}
On $[\rho_\lambda,\widehat{\rho}_\lambda]$ the solutions $\psi_3$ and $\psi_4$ have the representations

\begin{align*}
\psi_3(\rho,\lambda) &= c_{1,3}(\lambda)\psi_1(\rho,\lambda)+ c_{2,3}(\lambda)\psi_2(\rho,\lambda)\\
\psi_4(\rho,\lambda) &= c_{1,4}(\lambda)\psi_1(\rho,\lambda)+ c_{2,4}(\lambda)\psi_2(\rho,\lambda),
\end{align*}
with
\begin{align*}
	c_{1,3}(\lambda)&=\frac{\pi}{2}W(h_1(.,\lambda),b_2(.,\lambda))(\rho_\lambda)+\O(\langle\omega\rangle^{-1})=\O(\langle\omega\rangle^{0})
\\
c_{2,3}(\lambda)&= -\frac{\pi}{2}W(h_1(.,\lambda),b_1(.,\lambda))(\rho_\lambda)+\O(\langle\omega\rangle^{-1})=\O(\langle\omega\rangle^{0})
\end{align*}
and
\begin{align*}
	c_{1,4}(\lambda)&= \frac{\pi}{2}W(h_2(.,\lambda),b_2(.,\lambda))(\rho_\lambda)+\O(\langle\omega\rangle^{-1})=\O(\langle\omega\rangle^{0})
	\\
	c_{2,4}(\lambda)&=-\frac{\pi}{2}W(h_2(.,\lambda),b_1(.,\lambda))(\rho_\lambda)+\O(\langle\omega\rangle^{-1})=\O(\langle\omega\rangle^{0}).
\end{align*}
\end{lem}
\begin{proof}
We know the explicit representations
\begin{align*}
c_{1,3}(\lambda)&=\frac{W(\psi_3(.,\lambda),\psi_2(.,\lambda))}{W(\psi_1(.,\lambda),\psi_2(.,\lambda))}\\
c_{2,3}(\lambda)&=-\frac{W(\psi_3(.,\lambda),\psi_1(.,\lambda))}{W(\psi_{1}(.,\lambda),\psi_{2}(.,\lambda))}
\end{align*}
and so computing the connection coefficents reduces to calculating these Wronskians.
Evaluating $W(\psi_1(.,\lambda),\psi_2(.,\lambda))(\rho)$ at $\rho=0$ yields
$$
W(\psi_{1}(.,\lambda),\psi_{2}(.,\lambda))=W(b_1(.,\lambda),b_2(.,\lambda))=\frac{2}{\pi}
$$
while an evaluation at $\rho_\lambda$ yields
\begin{align*}
 W(\psi_3(.,\lambda),\psi_2(.,\lambda))&= W(
    h_1(.,\lambda), b_2(.,\lambda))(\rho_\lambda) [1+\O(\langle
    \omega\rangle^{-1})] \\
  &\quad +h_1(\rho_\lambda,\lambda)b_2(\rho_\lambda,\lambda)\O(\langle\omega\rangle^{0})+\O(\langle\omega\rangle^{-2}) \\
&=W(h_1(.,\lambda),b_2(.,\lambda))(\rho_\lambda)+\O(\langle\omega\rangle^{-1}).
\end{align*}
Hence, $$c_{1,3}(\lambda)= \frac{\pi}{2}W(h_1(.,\lambda),b_2(.,\lambda))(\rho_\lambda)+\O(\langle\omega\rangle^{-1})$$
and analogously, one computes the remaining coefficients.
\end{proof}
Analogously we can patch together the solutions of the free equation. To this end, let $\psi_{\mathrm{f}_1}$ and $\psi_{\mathrm{f}_2}$ be the solutions obtained from Lemma \ref{fundi near 0} in the case $V=0$ and, for notational convenience, let $h_1=\psi_{\mathrm{f}_3}$ and $h_2=\psi_{\mathrm{f}_4}$.
\begin{lem}\label{connectioncoeffree}
On $[\rho_\lambda,\widehat{\rho}_\lambda]$ the solutions $\psi_{\mathrm{f}_3}$ and $\psi_{\mathrm{f}_4}$ have the representations
\begin{align*}
\psi_{\mathrm{f}_3}(\rho,\lambda) &= c_{\mathrm{f}_{1,3}}(\lambda)\psi_{\mathrm{f}_1}(\rho,\lambda)+ c_{\mathrm{f}_{2,3}}(\lambda)\psi_{\mathrm{f}_2}(\rho,\lambda)
\\
\psi_{\mathrm{f}_4}(\rho,\lambda) &= c_{\mathrm{f}_{1,4}}(\lambda)\psi_{\mathrm{f}_1}(\rho,\lambda)+ c_{\mathrm{f}_{2,4}}(\lambda)\psi_{\mathrm{f}_2}(\rho,\lambda),
\end{align*}
with
\begin{align*}
	c_{\mathrm{f}_{1,3}}(\lambda)&=\frac{\pi}{2}W(h_1(.,\lambda),b_2(.,\lambda))(\rho_\lambda)+\O(\langle\omega\rangle^{-1})=\O(\langle\omega\rangle^{0})
\\
c_{\mathrm{f}_{2,3}}(\lambda)&= -\frac{\pi}{2}W(h_1(.,\lambda),b_1(.,\lambda))(\rho_\lambda)+\O(\langle\omega\rangle^{-1})=\O(\langle\omega\rangle^{0})
\end{align*}
and
\begin{align*}
	c_{\mathrm{f}_{1,4}}(\lambda)&= \frac{\pi}{2}W(h_2(.,\lambda),b_2(.,\lambda))(\rho_\lambda)+\O(\langle\omega\rangle^{-1})=\O(\langle\omega\rangle^{0})
	\\
	c_{\mathrm{f}_{2,4}}(\lambda)&=-\frac{\pi}{2}W(h_2(.,\lambda),b_1(.,\lambda))(\rho_\lambda)+\O(\langle\omega\rangle^{-1})=\O(\langle\omega\rangle^{0}).
\end{align*}
\end{lem}                                                      
Now, let $\chi: [0,1]\times \{z\in \C:0\leq \Re z\leq \frac{1}{4}\} \to
[0,1]$, $\chi_\lambda(\rho):=\chi(\rho,\lambda)$, be a smooth cutoff function that satisfies
$\chi_\lambda(\rho)=1$ for $\rho \in [0,\rho_\lambda]$, $
\chi_\lambda(\rho)=0$ for $\rho \in [\widehat{\rho}_\lambda,1] $, and
$|\partial_\rho^k\partial_\omega^\ell \chi_\lambda(\rho)|\leq
C_{k,\ell}\langle\omega\rangle^{k-\ell}$ for $k,\ell\in\mathbb N_0$.
We then obtain two global solutions to Eq.~\eqref{nofirstorder} by setting
\begin{align*}
v_1(\rho,\lambda):=&\chi_\lambda(\rho)[c_{1,3}(\lambda)\psi_1(\rho,\lambda)+c_{2,3}(\lambda)\psi_2(\rho,\lambda)]
+\left(1-\chi_\lambda(\rho)\right)\psi_3(\rho,\lambda)\\
v_2(\rho,\lambda):=& \chi_\lambda(\rho)[c_{1,4}(\lambda)\psi_1(\rho,\lambda)+c_{2,4}(\lambda)\psi_2(\rho,\lambda)]
+\left(1-\chi_\lambda(\rho)\right)\psi_4(\rho,\lambda).
\end{align*}
An evaluation at $\rho=1$ yields 
$$
W(v_1,v_2)=W(\psi_3,\psi_4)=2i.
$$
\subsection{Original equation}
We now return to our original equation
and our specified potential. By undoing our transformations, we obtain a fundamental system of equation \eqref{generalised spectral eq} which is given by
\begin{align*}
u_1(\rho,\lambda)=&\rho^{\frac{1-d}{2}}(1-\rho^2)^{-\frac{1}{4}-\frac{\lambda}{2}}v_1(\rho,\lambda)
\\
u_2(\rho,\lambda)=&\rho^{\frac{1-d}{2}}(1-\rho^2)^{-\frac{1}{4}-\frac{\lambda}{2}}v_2(\rho,\lambda).
\end{align*}
Observe, that neither $u_1$ nor $u_2$ remains bounded bounded at $\rho=0$. To construct a solution which is well behaved near $0$ we will need the following Lemma.
\begin{lem}\label{lem:roots of free c23}
We have $c_{2,3}(\lambda)\neq 0$ and $c_{\mathrm{f}_{2,3}}(\lambda)\neq 0$ for all $\lambda \in \C$ with $\Re \lambda \in  [0,\frac{1}{4}]$. 
\end{lem}
\begin{proof}
We argue by contradiction and assume that $c_{2,3}(\lambda)$ vanishes for some  $\lambda \in \C$ with $\Re \lambda \in  [0,\frac{1}{4}]$. Then $u_{1}(.,i\omega)\in H^1(\B^d_1)$ and so, $\lambda$ would be an eigenvalue of $\Lf.$ However, from Lemma \ref{lem:spectrum L} we infer that the point spectrum of $\Lf$, denoted by $\sigma_p(\Lf)$, satisfies $$\sigma_p(\Lf)\subset\{\lambda\in \C: \Re \lambda <0\} \cup \{1\},$$ which contradicts the assumption.
To prove the nonvanishing of $c_{\mathrm{f}_{2,3}}(\lambda)$ we argue similarly, and set 
$u_{\mathrm{f}_1}(\rho,\lambda):=\rho^{\frac{1-d}{2}}(1-\rho^2)^{-\frac{1}{4}-\frac{\lambda}{2}}v_{\mathrm{f}_1}(\rho,\lambda)$.
This function then solves
\begin{align}\label{Eq:nopot}
(1-\rho^2)u_1''(\rho)+\left(\frac{d-1}{\rho}-(2\lambda+d)\rho\right)u_1'(\rho)-\left(\lambda\left(\lambda+d-1\right)+d\frac{d-2}{4}\right)u_1(\rho)=0
\end{align}
for $\rho \in (0,1)$ and the vanishing of $c_{\mathrm{f}_{2,3}}(\lambda)$ implies
$$u_{\mathrm{f}_1}(.,\lambda) \in H^1(\B^d_1).$$
Consequently, if $\lambda$ is such that $c_{\mathrm{f}_{2,3}}(\lambda)$ vanishes, then $\lambda$ is an eigenvalue of $\Lf_0$. However, as $\Lf_0$ generates a $C_0$-semigroup, the spectrum of $\Lf_0$ is confined to the closed left halfplane.
Thus, we are left with checking the claim on the imaginary axis. Assume that $c_{\mathrm{f}_{2,3}}(i \omega)$ vanishes for some $\omega\in \R$. The Frobenius indices of Eq.~\eqref{Eq:nopot} are given by $\{0,\frac{d-2}{2}\}$ at $0$ and $\{0,\frac{1}{2}-\lambda\}$ at 1 and so, from  $u_{\mathrm{f}_1}(.,i\omega)\in H^1(\B^d_1)$, we see that $u_{\mathrm{f}_1}(.,i\omega)$ is in fact a smooth function on $[0,1]$.
Setting $z=\rho^2$ and $v(z)=u(\sqrt{z})$, transforms equation \eqref{Eq:nopot} into
\begin{align}\label{eq:hypergeo2}
z(1-z)v''(z)&+\frac{d}{2} v'(z)-\left(i\omega+\frac{d+1}{2}\right)z v'(z) -\frac{1}{4}\left(i\omega\left(i\omega+d-1\right)+d\frac{d-2}{4}\right)v_1(z)=0.
\end{align}
This is now again a hypergeometric differential equation,
with $a=\frac{2i\omega+d-2}{4}, b=\frac{2i\omega+d}{4},$ and $c=\frac{d}{2}$.
A fundamental system of this equation near $0$, is therefore given by 
\begin{align*}
f_1(z):=\, _2F_1(a,b;c;z)
\end{align*}
and a second solution which diverges at $\rho=0$,
while near 1, a fundamental system  is given by 
\begin{align*}
f_2(z):=&\,_2F_1(a,b;a+b+1-c;1-z)\\
f_3(z):=& (1-z)^{\frac{1}{2}-i\omega} \,_2F_1(c-a,c-b;c-a-b+1;1-z).
\end{align*}
Therefore, for a solution of Eq.~\eqref{eq:hypergeo2} to be smooth, it has to be a multiple of $f_1$.
Let $c_2,c_3\in \C $ be such that
\begin{align*}
f_1(z)=c_2 f_2(z)+c_3 f_3(z)
\end{align*}
and note that the smoothness of $u_{\mathrm{f}_1}$ necessitates the vanishing of $c_3$. Fortunately, $c_3$ is explicitly given as
\begin{align*}
c_3=\frac{\Gamma(c)\Gamma(a+b-c+1)}{\Gamma(a) \Gamma(b)}.
\end{align*}
So, for $c_3$ to vanish, either $a$ or $b$ would have to be a pole of $\Gamma$ which is impossible on the imaginary axis.
\end{proof}

Thanks to the last lemma, we obtain a solution $u_0$ which does not diverge at $\rho=0$ by setting
$u_0=u_2-\frac{c_{2,4}}{c_{2,3}}u_1$.
Next, we denote the free solution with the same behavior as $u_j$ by $u_{\mathrm{f}_j}$ for $j=0,1,2$
and note that
\begin{align*}
W(u_1,u_0)=W(u_{\mathrm f_1},u_{\mathrm f_0})=2i\rho^{1-d}(1-\rho^2)^{-\frac{1}{2}-\lambda}.
\end{align*}
Observe now, that neither $u_0$ nor $u_1$ is in $H^1(\B^d_1)$. The function $u_1$ fails to be in the space $H^1(\B^d_1)$ due to its divergent behavior at $0$, whereas $u_0'$ is not square integrable at 1. It follows that a unique solution $u\in H^1(\B^d_1)$ of Eq.~\eqref{generalised spectral eq} is given by 
\begin{align*}
u(\rho,\lambda)=u&_0(\rho,\lambda) \int_\rho^{1}\frac{u_1(s,\lambda)}{W(u_1(.,\lambda),u_0(.,\lambda))(s)}\frac{F_\lambda(s)}{1-s^2} d s\\
&+u_1(\rho,\lambda) \int_0^\rho\frac{u_0(s,\lambda)}{W(u_1(.,\lambda),u_0(.,\lambda))(s)}\frac{F_\lambda(s)}{1-s^2} d s.
\end{align*}
Hence, we see that the Green function of Eq. \eqref{generalised spectral eq}
is given by
\begin{align*}
G(\rho,s,\lambda)=\frac{s^{d-1}(1-s^2)^{-\frac{1}{2}+\lambda}}{2i}\begin{cases}
u_0(\rho,\lambda)u_1(s,\lambda) \text{ if } \rho \leq s
\\
u_1(\rho,\lambda)u_0(s,\lambda) \text{ if } \rho \geq s.
\end{cases}
\end{align*}
Similarly, we obtain a solution $u_\mathrm{f}$ of the free equation which is of the same form and an associated Green function 
\begin{align*}
G_{\mathrm{f}}(\rho,s,\lambda)=\frac{s^{d-1}(1-s^2)^{-\frac{1}{2}+\lambda}}{2i}\begin{cases}
u_{\mathrm{f}_0}(\rho,\lambda)u_{\mathrm{f}_1}(s,\lambda) \text{ if } \rho \leq s
\\
u_{\mathrm{f}_1}(\rho,\lambda)u_{\mathrm{f}_0}(s,\lambda) \text{ if } \rho \geq s.
\end{cases}
\end{align*}
This implies the following decomposition of the Green function.
\begin{lem}\label{lem:decomp}
We can decompose  $G$ according to
$$G(\rho,s,\lambda)=G_{\mathrm{f}}(\rho,s,\lambda) +\frac{\rho^{\frac{1-d}{2}}}{2i}(1-\rho^2)^{-\frac{1}{4}-\frac{\lambda}{2}}s^{\frac{d-1}{2}}(1-s^2)^{-\frac{3}{4}+\frac{\lambda}{2}}\sum_{j=1}^8 G_{j}(\rho,s,\lambda)$$
with
\begin{align*}
G_1(\rho,s,\lambda)&=1_{\R_+}(s-\rho)\chi_\lambda(s)b_{1}(\rho,\lambda)[b_1(s,\lambda)\alpha_1(\rho,s,\lambda)+b_2(s,\rho)\alpha_2(\rho,s,\lambda)
\\
&\quad+\O(s^{\frac{7-d}{2}}\langle\omega\rangle^{1-\frac{d}{2}})]
\\
G_2(\rho,s,\lambda)&=1_{\R_+}(s-\rho)\chi_\lambda(\rho)(1-\chi_\lambda(s))b_1(\rho,\lambda)h_1(s,\lambda)\beta(\rho,s,\lambda)
\\
G_3(\rho,s,\lambda)&=1_{\R_+}(s-\rho)(1-\chi_\lambda(\rho))(1-\chi_\lambda(s))h_{1}(\rho,\lambda)h_1(s,\lambda)\gamma_1(\rho,s,\lambda)
\\
G_4(\rho,s,\lambda)&=1_{\R_+}(s-\rho)(1-\chi_\lambda(\rho))(1-\chi_\lambda(s))
h_2(\rho,\lambda)h_1(s,\lambda)\gamma_2(\rho,s,\lambda)
\\
G_5(\rho,s,\lambda)&=1_{\R_+}(\rho-s)\chi_\lambda(\rho)[b_{1}(\rho,\lambda)\alpha_1(s,\rho,\lambda)+b_2(\rho,\lambda)\alpha_2(s,\rho,\lambda)+\O(\rho^{\frac{7-d}{2}}\langle\omega\rangle^{1-\frac{d}{2}})]
\\
&\quad \times b_1(s,\lambda)
\\
G_6(\rho,s,\lambda)&=1_{\R_+}(\rho-s)(1-\chi_\lambda(\rho))\chi_\lambda(s)h_1(\rho,\lambda)b_1(s,\lambda)\beta(s,\rho,\lambda)
\\
G_7(\rho,s,\lambda)&=1_{\R_+}(\rho-s)(1-\chi_\lambda(\rho))(1-\chi_\lambda(s))h_{1}(\rho,\lambda)h_1(s,\lambda)\gamma_1(s,\rho,\lambda)
\\
G_8(\rho,s,\lambda)&=1_{\R_+}(\rho-s)(1-\chi_\lambda(\rho))(1-\chi_\lambda(s)h_1(\rho,\lambda)h_2(s,\lambda)\gamma_2(s,\rho,\lambda)
\end{align*}
where
\begin{align*}
\alpha_j(\rho,s,\lambda)&=\O(\langle\omega\rangle^{-1})+\O(\rho^2\langle\omega\rangle^{0})+\O(s^2\langle\omega\rangle^{0})+\O(\rho^2s^2\langle\omega\rangle^{0})  
\\
\beta(\rho,s,\lambda)&= 
\O(\langle\omega\rangle^{-1})+\O(s^0(1-s)\langle\omega\rangle^{-1})+\O(\rho^2\langle\omega\rangle^{0})+\O(\rho^2s^0(1-s)\langle\omega\rangle^{-1})
\\
\gamma_j(\rho,s,\lambda)&=
\O(\langle\omega\rangle^{-1})+\O(\rho^0(1-\rho)\langle\omega\rangle^{-1})+\O(s^0(1-s)\langle\omega\rangle^{-1})
\\
&\quad+\O(\rho^0(1-\rho)s^0(1-s)\langle\omega\rangle^{-2})
\end{align*}
for $j=1,2$.
\end{lem}
\begin{proof}
We start by recalling that
\begin{align*}
u_0(\rho)&=\rho^{\frac{1-d}{2}}(1-\rho^2)^{-\frac{1}{4}-\frac{\lambda}{2}}\chi_\lambda(\rho)\left(c_{1,4}(\lambda)-c_{1,3}(\lambda)\frac{c_{2,4}(\lambda)}{c_{2,3}(\lambda)}\right)\psi_1(\rho,\lambda)
\\
&\quad+\rho^{\frac{1-d}{2}}(1-\rho^2)^{-\frac{1}{4}-\frac{\lambda}{2}}(1-\chi_\lambda(\rho))\left[\psi_4(\rho,\lambda)-\frac{c_{2,4}(\lambda)}{c_{2,3}(\lambda)}\psi_3(\rho,\lambda)\right]
\end{align*}
and
\begin{align*}
u_1(\rho)&=\rho^{\frac{1-d}{2}}(1-\rho^2)^{-\frac{1}{4}-\frac{\lambda}{2}}\left[\chi_\lambda(\rho)(c_{1,3}(\lambda)\psi_1(\rho,\lambda)+c_{2,3}(\lambda)\psi_2(\rho,\lambda))
+\left(1-\chi_\lambda(\rho)\right)\psi_3(\rho,\lambda)\right].
\end{align*}
Furthermore, the solutions $u_{\mathrm{f}_j}$ are of the same form with $\psi_j$ and $c_{j,k}$ replaced by $\psi_{\mathrm{f}_j}$ and $c_{\mathrm{f}_{j,k}}$.
Moreover, we also record the following identities
\begin{align*}
c_{j,k}(\lambda)-c_{\mathrm{f}_{j,k}}(\lambda)&=\O(\langle\omega\rangle^{-1}),
\\
c_{1,3}(\lambda)\frac{c_{2,4}(\lambda)}{c_{2,3}(\lambda)}-c_{\mathrm{f}_{1,3}}(\lambda)\frac{c_{\mathrm{f}_{2,4}}(\lambda)}{c_{\mathrm{f}_{2,3}}(\lambda)}&=\O(\langle\omega\rangle^{-1}),
\end{align*}
and
\begin{align*}
\psi_1(\rho,\lambda)-\psi_{\mathrm{f}_1}(\rho,\lambda)=b_1(\rho,\lambda)\O(\rho^2\langle\omega\rangle^{0}).
\end{align*}
Since similar identities hold for 
$$
\psi_j-\psi_{\mathrm{f}_j}
$$
for $j=2,3,4$, the claim follows from a 
repeated 
usage of the identity
\begin{align*}
ab-cd=a(b-d)+d(a-c)
\end{align*}
and a straightforward computation.
\end{proof}
To cast the $G_j$ in a more manageable form, we also derive corresponding symbol forms.
\begin{lem}\label{lem:kernel symb}
The functions $\widetilde{G}_j$ defined as 

$$
\widetilde{G}_j(\rho,s,\lambda):=\frac{\rho^{\frac{1-d}{2}}}{2i}(1-\rho^2)^{-\frac{1}{4}-\frac{\lambda}{2}}s^{\frac{d-1}{2}}(1-s^2)^{-\frac{3}{4}+\frac{\lambda}{2}} G_{j}(\rho,s,\lambda)$$
satisfy
\begin{align*}
\widetilde{G}_1(\rho,s,\lambda)&=1_{\R_+}(s-\rho)\chi_\lambda(s)(1-\rho^2)^{-\frac{1}{4}-\frac{\lambda}{2}}(1-s^2)^{-\frac{3}{4}+\frac{\lambda}{2}}\O(\rho^0s\langle\omega\rangle^{-1})
\\
\widetilde{G}_2(\rho,s,\lambda)&=1_{\R_+}(s-\rho)\chi_\lambda(\rho)(1-\chi_\lambda(s))(1-\rho^2)^{-\frac{1}{4}-\frac{\lambda}{2}}
\\
& \quad\times s^{\frac{d-1}{2}}(1-s)^{-\frac{1}{2}+\lambda}\O(\rho^{0}\langle\omega\rangle^{\frac{d-3}{2}})\widehat{\beta}(\rho,s,\lambda)
\\
\widetilde{G}_3(\rho,s,\lambda)&=1_{\R_+}(s-\rho)(1-\chi_\lambda(\rho))(1-\chi_\lambda(s))\rho^{\frac{1-d}{2}}(1+\rho)^{\frac{1}{2}-\lambda}s^{\frac{d-1}{2}}(1-s)^{-\frac{1}{2}+\lambda}
\\
& \quad\times\O(\langle\omega\rangle^{-1})\widehat{\gamma}_1(\rho,s,\lambda)
\\
\widetilde{G}_4(\rho,s,\lambda)&=1_{\R_+}(s-\rho)(1-\chi_\lambda(\rho))(1-\chi_\lambda(s))\rho^{\frac{1-d}{2}}(1-\rho)^{\frac{1}{2}-\lambda}s^{\frac{d-1}{2}}(1-s)^{-\frac{1}{2}+\lambda}
\\
& \quad\times\O(\langle\omega\rangle^{-1})\widehat{\gamma}_2(\rho,s,\lambda)
\\
\widetilde{G}_5(\rho,s,\lambda)&=1_{\R_+}(\rho-s)\chi_\lambda(\rho)(1-\rho^2)^{-\frac{1}{4}-\frac{\lambda}{2}}(1-s^2)^{-\frac{3}{4}+\frac{\lambda}{2}}\O(\rho^0s\langle\omega\rangle^{-1})
\\
\widetilde{G}_6(\rho,s,\lambda)&=1_{\R_+}(\rho-s)(1-\chi_\lambda(\rho))\chi_\lambda(s)\rho^{\frac{1-d}{2}}(1+\rho)^{\frac{1}{2}-\lambda}
\\
& \quad\times
(1-s^2)^{-\frac{3}{4}+\frac{\lambda}{2}}\O(s^{d-1}\langle\omega\rangle^{\frac{d-3}{2}})\widehat{\beta}(s,\rho,\lambda)
\\
\widetilde{G}_7(\rho,s,\lambda)&=1_{\R_+}(s-\rho)(1-\chi_\lambda(\rho))(1-\chi_\lambda(s))\rho^{\frac{1-d}{2}}(1+\rho)^{\frac{1}{2}-\lambda}s^{\frac{d-1}{2}}(1-s)^{-\frac{1}{2}+\lambda}
\\
& \quad\times\O(\langle\omega\rangle^{-1})\widehat{\gamma}_1(s,\rho,\lambda)
\\
\widetilde{G}_8(\rho,s,\lambda)&=1_{\R_+}(s-\rho)(1-\chi_\lambda(\rho))(1-\chi_\lambda(s))\rho^{\frac{1-d}{2}}(1+\rho)^{\frac{1}{2}-\lambda}s^{\frac{d-1}{2}}(1+s)^{-\frac{1}{2}+\lambda}
\\
& \quad\times\O(\langle\omega\rangle^{-1})\widehat{\gamma}_2(s,\rho,\lambda)
\end{align*}
with
\begin{align*}
\widehat{\beta}(\rho,s,\lambda)&=[1+\O(s^{-1}(1-s)\langle\omega\rangle^{-1})]\beta(\rho,s,\lambda)
\\
\widehat{\gamma}_j(\rho,s,\lambda)&=[1+\O(\rho^{-1}(1-\rho)\langle\omega\rangle^{-1})][1+\O(s^{-1}(1-s)\langle\omega\rangle^{-1})]\gamma_j(\rho,s,\lambda).
\end{align*}
\end{lem}
\begin{proof}
This is a direct consequence of Lemmas \ref{Lem:hom near 1} and \ref{fundi near 0}.
\end{proof}
To proceed, let $\ff \in C^2$ and set $\widetilde{\ff}=(\I -\Pf)\ff$, where $\Pf$ is the spectral projection associated to the eigenvalue $1$ of $\Lf$.
By making use of Laplace inversion and Lemma \ref{lem:decomp}, the first component $\Sf$ applied to $ \widetilde{\ff}$ is then explicitly given by
\begin{align}\label{Eq:rep semigroup}
[\Sf(\tau)\widetilde{\ff}]_1(\rho)=[\Sf_0(\tau)\widetilde{\ff}]_1(\rho)+\frac{1}{2\pi i}\sum_{j=1}^8\lim_{N \to \infty}\int_{\varepsilon-iN}^{\varepsilon+iN}
e^{\lambda \tau}\int_0^1\widetilde{G}_j(\rho,s,\lambda)F_\lambda(s) ds d\lambda
\end{align}
for $\varepsilon>0$
and $F_\lambda(s)=s\widetilde{f}_1'(s)+(\lambda+\frac{d}{2})\widetilde{f}_1(s)+\widetilde{f}_2(s)$.
\section{Strichartz estimates}
In this section, we will always assume $d\geq4$, as in the case $d=3$, the Strichartz estimates were already shown in \cite{Don17}. This restriction is to avoid the Strichartz pair $(2,\infty)$, which we could just as well treat here. However, as this would slightly  complicate matters (mostly the notation), we choose not do so.
The representation \eqref{Eq:rep semigroup} of $\Sf$ is one of the key ingredients to prove Theorem \ref{thm:strichartz}, i.e., to prove Strichartz estimates for the full semigroup $\Sf$.
To accomplish this, we start by defining operators $T_j$ for $j=1,\dots,8$ as
\begin{align*}
T_{j}(\tau)f(\rho):=\frac{1}{2\pi}\lim_{\varepsilon \to 0^+}\lim_{N\to \infty}\int_{\varepsilon-iN}^{\varepsilon+iN} 
e^{(\varepsilon+i\omega) \tau}\int_0^1\widetilde{G}_j(\rho,s,\varepsilon+i\omega)f(s) ds d\omega
\end{align*}
for $f\in C^2(\overline{\B^d_1})$.
\begin{lem} \label{lem: interchange}
The operators $T_j$ satisfy
\begin{align*}
T_{j}(\tau)f(\rho)=\frac{1}{2\pi} \int_0^1 \int_\R
e^{i\omega \tau} \widetilde{G}_j(\rho,s,i\omega)f(s) d\omega  ds
\end{align*}
for all $f\in C^2(\overline{\B^d_1})$.
\end{lem}
\begin{proof}
From Lemma \ref{lem:kernel symb}, we obtain the estimate \begin{align*}
|\widetilde{G}_j(\rho,s,\lambda)|\lesssim \rho^{\frac{1-d}{2}}\langle\omega\rangle^{-2}(1-s)^{-\frac{1}{2}}
\end{align*}
for $j=1,\dots 8$, $\Re\lambda\in [0,\frac{1}{4}]$, and $\rho\in (0,1)$.
As a consequence, the claim follows from the Fubini-Tonelli Theorem and dominated convergence.
\end{proof}
\subsection{Technical Lemmas}
To establish Theorem \ref{thm:strichartz}, we will have to make use of a couple of technical Lemmas. We start with some basic results on oscillatory integrals.

\begin{lem}\label{osci1}
	Let $\alpha >0$. Then 
	$$
	\left|\int_\R e^{i a\omega}\O(\langle\omega\rangle^{-(1+\alpha)}) d \omega \right|\lesssim \langle a\rangle^{-2}
	$$
	 for all $a\in \R$. 
\end{lem}

\begin{proof}
	Since the integral is absolutely convergent the claim follows by doing two integrations by parts.
\end{proof}
\begin{lem}\label{osci2}
	Let $\alpha >0$. Then,
	\begin{align*}
	\left|\int_\R e^{i a\omega}\O(\langle\omega\rangle^{-\alpha}) d \omega \right|\lesssim |a|^{1-\alpha}\langle a\rangle^{-2}
	\end{align*}
	for all  $a\in \R\setminus\{0\}$.
\end{lem}
\begin{proof}
	See Lemma 4.2 in \cite{DonRao20}.
\end{proof}
\begin{lem}\label{osci3}
The estimate
	\begin{align*}
	\left|\int_\R e^{i a\omega}(1-\chi_\lambda(\rho))\O\left(\rho^{-n}\langle\omega\rangle^{-(n+1)}\right) d \omega \right|\lesssim \langle a\rangle^{-2}
	\end{align*}
	holds for all $n\geq 1$, $\rho \in (0,1)$, and $a\in \R \setminus\{0\}$.
\end{lem}
\begin{proof}
	This can be proven in the same manner as Lemma 4.3 in \cite{DonRao20}
\end{proof}
Finally, we need one more lemma on oscillatory integrals.
\begin{lem}\label{osci4}
The estimate
	$$
	\left|\int_\R e^{i a\omega}(1-\chi_\lambda(\rho))\O\left(\rho^{-n}\langle\omega\rangle^{-n}\right)d\omega\right|\lesssim|a|^{-1}\langle a\rangle^{-2}
	$$
holds for all $n \geq 2$, $\rho \in (0,1)$, and $a\in \R \setminus\{0\}$.
\end{lem}

\begin{proof}
This can be proven as Lemma 4.4 in \cite{DonRao20}.
\end{proof}
We will also utilise of the following technical lemmas.
\begin{lem}\label{teclem1}
The estimate
\begin{align*}
\left\||.|^{-1}f\right\|_{L^{2}(\B^d_1)}\lesssim \|f\|_{H^1(\B^d_1)}
\end{align*}
holds for all $f\in C^1(\overline{\B^d_1})$. 
\end{lem}
\begin{proof}
An integration by parts shows that
\begin{align*}
\left\||.|^{-1}f\right\|_{L^{2}(\B^d_1)}^2&=\int_0^1  |f(\rho)|^2 \rho^{d-3} d\rho \lesssim |f(1)|^2+\int_0^1  |f'(\rho)| |f(\rho)| \rho^{d-2} d\rho
\\
&
\lesssim \frac{1}{\varepsilon} \|f\|_{\dot{H}^1(\B^d_1)}+\varepsilon \left\||.|^{-1}f\right\|_{L^{2}(\B^d_1)}^2
\end{align*}
and so the claim follows from an appropriate choice of $\varepsilon$.
\end{proof}
\begin{lem}\label{teclem2}
The estimates
\begin{align*}
\||.| f\|_{L^{\frac{2d}{d-3}}(\B^d_1)}\lesssim \|f\|_{H^1(\B^d_1)}
\end{align*}
and 
\begin{align*}
\left\|\int_\rho^1 s f'(s)ds\right\|_{L^{\frac{2d}{d-3}}_\rho(\B^d_1)} \lesssim \|f\|_{H^1(\B^d_1)}
\end{align*}
hold for all $f\in C^1(\overline{\B^d_1})$. 
\end{lem}
\begin{proof}
Let $n=\lceil\frac{d}{2}\rceil$. Then, from the Sobolev embedding $H^1(\B^n)\hookrightarrow L^{\frac{2n}{n-2}}(\B^n_1)$ we infer that
\begin{align*}
||.|f\|_{L^{\frac{2d}{d-3}}(\B^d_1)}&=\left(\int_0^1 \rho^{n-1}\left(|f(\rho)|\rho^{1+(d-n)\frac{d-3}{2d}}\right)^{\frac{2d}{d-3}} d\rho \right)^{\frac{d-3}{2d}}=\||.|^{1+(d-n)\frac{d-3}{2d}}
f\|_{L^{\frac{2d}{d-3}}(\B^n_1)}
\\
&\leq \||.|^{1+(d-n)\frac{d-3}{2d}}
f\|_{L^{\frac{2n}{n-2}}(\B^n_1)} \lesssim \||.|^{1+(d-n)\frac{d-3}{2d}}
f\|_{H^1(\B^n_1)}.
\end{align*}
Further,
\begin{align*}
 \||.|^{1+(d-n)\frac{d-3}{2d}}
f\|_{L^2(\B^n_1)}= \int_0^1|f(\rho)|^2 \rho^{2+(d-n)\frac{d-3}{d}} \rho^{n-1} d\rho\leq \int_0^1|f(\rho)|^2 \rho^{d-1} d\rho =\|f\|_{L^2(\B^d_1)}
\end{align*}
as well as
\begin{align*}
 \||.|^{1+(d-n)\frac{d-3}{2d}}
f\|_{\dot{H}^1(\B^n_1)}&\lesssim\|f\|_{\dot{H}^1(\B^d_1)}+\||.|^{(d-n)\frac{d-3}{2d}}
f\|_{L^2(\B^n_1)}
\\
&\lesssim \|f\|_{H^1(\B^d_1)}+ \||.|^{-1}f\|_{L^2(\B^d_1)} \lesssim\|f\|_{H^1(\B^d_1)},
\end{align*}
by Lemma \ref{teclem1}. For the second one we argue similarly to obtain that
\begin{align*}
\left\|\int_\rho^1 s f'(s)ds\right\|_{L^{\frac{2d}{d-3}}_\rho(\B^d_1)} &\lesssim \left\|\rho^{(d-n)\frac{d-3}{2d}}\int_\rho^1 s f'(s)ds\right\|_{H^1_\rho(\B^n_1)}.
\end{align*}
Observe that
\begin{align*}
\left\|\rho^{(d-n)\frac{d-3}{2d}}\int_\rho^1 s f'(s)ds\right\|_{L^2_\rho(\B^n_1)}\lesssim\left\|\rho^{\frac{1-n}{2}+\frac{1}{4}}\int_\rho^1 s^{\frac{d-1}{2}} f'(s)ds\right\|_{L^2_\rho(\B^n_1)}\lesssim \|f\|_{H^1(\B^d_1)}
\end{align*}
and 
\begin{align*}
\left\|\rho^{(d-n)\frac{d-3}{2d}}\int_\rho^1 s f'(s)ds\right\|_{\dot{H}^1_\rho(\B^n_1)}
&\lesssim \|f\|_{H^1(\B^d_1)}
+\left\|\rho^{\frac{1-n}{2}-\frac{3}{4}}\int_\rho^1 s^{\frac{d-1}{2}} f'(s)ds\right\|_{L^2_\rho(\B^n_1)}.
\end{align*}
Lastly, we use the Cauchy Schwarz inequality to estimate that
\begin{align*}
\left\|\rho^{\frac{1-n}{2}-\frac{3}{4}}\int_\rho^1 s^{\frac{d-1}{2}} f'(s)ds\right\|_{L^2_\rho(\B^n_1)}
&\leq \left\|\rho^{\frac{1-n}{2}-\frac{7}{20}}\int_\rho^1 s^{\frac{d-1}{2}-\frac{2}{5}} f'(s)ds\right\|_{L^2_\rho(\B^n_1)} 
\\
& \lesssim \|f\|_{H^1(\B^d_1)}\||.|^{\frac{1-n}{2}-\frac{7}{20}}\|_{L^2(\B^n_1)}\left(\int_0^1 s^{-\frac{4}{5}}ds\right)^{\frac{1}{2}}
\\
&\lesssim \|f\|_{H^1(\B^d_1)}.
\end{align*}
\end{proof}
Moreover, we will make use of the following two Lemmas, proven in \cite{DonWal22a}.
\begin{lem} [{\cite[Lemma 5.6]{DonWal22a}}] \label{teclem4}
	Let $\alpha \in (0,1)$. Then the estimate
	\begin{align*}
	\int_0^1s^{\alpha-1}|a+\log(1\pm s)|^{-\alpha}| ds\lesssim |a|^{-\alpha}
	\end{align*}
	 holds for all $a\in \R\setminus\{0\}$.
\end{lem}
\begin{lem} [{\cite[Lemma 5.7]{DonWal22a}}]\label{teclem5}
The estimate
	\begin{align*}
	\int_0^1 s^{-\frac{1}{2}}\left|a\pm \frac{1}{2}\log(1-s^2) \right|^{-\frac{1}{2}} ds\lesssim |a|^{-\frac{1}{2}}
	\end{align*}
	holds for all $a\in \R\setminus\{0\}$.
\end{lem}
With these technical results out of the way, we continue by establishing bounds on the operators $T_j$.
\begin{lem}\label{lem:strichartz 1}
Let $p\in [2,\infty]$ and $ q\in [\frac{2d}{d-2},\frac{2d}{d-3}]$ be such that
$\frac{1}{p}+\frac{d}{q}=\frac{d}{2}-1$.
Then the operators $T_j$ satisfy
\begin{align*}
\|T_j(.)f\|_{L^p(\R_+)L^q(\B^d_1)}\lesssim \|f\|_{L^2(\B^d_1)}
\end{align*}
for all $f \in C^1(\overline{\B^d_1})$.
\end{lem}
\begin{proof}
We start off with $T_1$ which satisfies
\begin{align*}
T_1(\tau)f(\rho)=
\int_\rho^{\rho_1} f(s) \int_\R e^{i\omega\tau}\chi_{i\omega}(s)(1-\rho^2)^{-\frac{1}{4}-\frac{i\omega}{2}}(1-s^2)^{-\frac{3}{4}-\frac{i\omega}{2}}\O(\rho^0s\langle\omega\rangle^{-1})  d\omega ds
\end{align*}
for some $\rho_1<1$, thanks to Lemmas \ref{lem:kernel symb} and \ref{lem: interchange} and the fact that $\chi_{i \omega}(s)$ is supported away from $1$. 
Using that 
$$\chi_{i \omega}(s)\O( \rho^0s\langle\omega\rangle^{-1})=\chi_{i \omega}(s) \O(\rho^{\frac{3-d}{2}+\frac{1}{8}} s^{\frac{d-1}{2}-\frac{1}{4}}\langle\omega\rangle^{-1-\frac{1}{8}})$$ for $\rho \leq s$ and Lemma \ref{osci1} yields
\begin{align*}
|T_1(\tau)f(\rho)|\lesssim&\rho^{\frac{3-d}{2}+\frac{1}{8}}\int_0^{\rho_1} s^{\frac{d-1}{2}-\frac{1}{4}}f(s)\langle\tau-\frac{1}{2}\ln(1-\rho^2)+\frac{1}{2}\ln(1-s^2)\rangle^{-2} ds
\\
&\langle\tau\rangle^{-2}\rho^{\frac{3-d}{2}+\frac{1}{8}} \int_0^{1} s^{\frac{d-1}{2}-\frac{1}{4}}|f(s)|ds
\\
\lesssim& \langle\tau\rangle^{-2}\rho^{\frac{3-d}{2}+\frac{1}{8}}\|f\|_{L^2(\B^d_1)}.
\end{align*}
Since,
\begin{align*}
\||.|^{\frac{3-d}{2}+\frac{1}{8}}\|_{L^{\frac{2d}{d-3}}(\B^d_1)}=\left(\int_0^1 \rho^{-1+\frac{4d}{d-3}} d\rho \right)^{\frac{d-3}{2d}}\lesssim 1
\end{align*}
we conclude that 

\begin{align*}
\|T_1(.)f\|_{L^\infty(\R_+)L^{\frac{2d}{d-2}}(\B^d_1)}+
\|T_1(.)f\|_{L^2(\R_+)L^{\frac{2d}{d-3}}(\B^d_1)} \lesssim \|f\|_{L^2(\B^d_1)}.
\end{align*}
The estimate 
\begin{align*}
\|T_1(.)f\|_{L^p(\R_+)L^{q}(\B^d_1)}
&\lesssim \|f\|_{L^2(\B^d_1)}
\end{align*}
for general admissible pairs $(p,q)$ then follows from the interpolation argument done in the proof of Lemma \ref{lem: free strichartz}.
Slightly adapting these arguments allows us to also obtain the desired estimates for $T_5$.
Next, we turn to $T_2$ which is given by
\begin{align*}
T_2(\tau)f(\rho)&=\int_{\rho}^1 f(s)\int_\R e^{i\omega\tau}\chi_{i\omega}(\rho)(1-\chi_{i\omega}(s))(1-\rho^2)^{-\frac{1}{4}-\frac {i\omega}{2}}
\\
& \quad\times s^{\frac{d-1}{2}}(1-s)^{-\frac{1}{2}+i\omega}\O(\rho^{\frac{3-d}{2}+\frac{d-3}{4d}}\langle\omega\rangle^{\frac{d-3}{4d}})\widehat{\beta}(\rho,s,i \omega) d\omega ds
\end{align*} 
Since $\widehat{\beta}$ behaves like $\O(\langle\omega\rangle^{-1}) $ in $\omega$,
we can use Lemma \ref{osci2} with $\alpha=1-\frac{d-3}{4d}$ to obtain the estimate
\begin{align*}
\|T_2(\tau)f\|_{L^{\frac{2d}{d-3}}(\B^d_1)}\lesssim& \bigg\|\rho^{\frac{3-d}{2}+\frac{d-3}{4d}}\int_0^1\langle\tau+\ln(1-s)\rangle^{-2}
\\
&\times |\tau-\frac{1}{2}\ln(1-\rho^2)+\ln(1-s)|^{-\frac{d-3}{4d}}
(1-s)^{-\frac{1}{2}} s^{\frac{d-1}{2}}f(s)ds \bigg\|_{L^{\frac{2d}{d-3}}(\B^d_1)}
\\
\lesssim&\int_0^1 (1-s)^{-\frac{1}{2}} s^{\frac{d-1}{2}}|f(s)|\langle\tau+\ln(1-s)\rangle^{-2}
\\
&\times\left\|\rho^{\frac{3-d}{2}+\frac{d-3}{4d}} |\tau-\frac{1}{2}\ln(1-\rho^2)+\ln(1-s)|^{-\frac{d-3}{4d}}
\right\|_{L^{\frac{2d}{d-3}}(\B^d_1)}ds
\\
=&\int_0^1 (1-s)^{-\frac{1}{2}} s^{\frac{d-1}{2}}|f(s)|\langle\tau+\ln(1-s)\rangle^{-2}
\\
& \times \int_0^1
\left(\rho^{-\frac{1}{2}}|\tau-\frac{1}{2}\ln(1-\rho^2)+\ln(1-s)|^{-\frac{1}{2}} d\rho \right)^{\frac{d-3}{2d}}ds
\\
\lesssim &\int_0^1 (1-s)^{-\frac{1}{2}} s^{\frac{d-1}{2}}|f(s)|\langle\tau+\ln(1-s)\rangle^{-2}|\tau+\ln(1-s)|^{\frac{3-d}{4d}} ds
\end{align*}
by additionally employing Lemma \ref{teclem5}.
By now making the change of coordinates $s=1-e^{-y}$
and using Young's inequality, we derive that
\begin{align*}
\|T_2(.)f\|_{L^2(\R_+)L^{\frac{2d}{d-3}}(\B^d_1)}\lesssim& \left\|\int^{\infty}_0\langle\tau-y\rangle^{-2}|\tau-y|^{\frac{3-d}{4d}}e^{-\frac{1}{2}y} (1-e^{-y})^{\frac{d-1}{2}}|f(1-e^{-y})| dy\right\|_{L^2(\R_+)}
\\
\lesssim&
\left(\int_0^\infty e^{-y} (1-e^{-y})^{d-1}|f(1-e^{-y})|^2 dy \right)^2\int_\R \langle\tau\rangle^{-2}|\tau|^{-\frac{3-d}{4d}} d\tau
\\
\lesssim& \|f\|_{L^2(\B^d_1)}.
\end{align*}
As the other endpoint estimate follows likewise, the general one follows once more from interpolation.
$T_6$ can be bounded analogously by using Lemma \ref{teclem4} instead of Lemma \ref{teclem5} and we move on to $T_3$, which is given by
\begin{align*}
T_3(\tau)f(\rho)&=\int_0^1\rho^{\frac{1-d}{2}} 1_{[0,\infty)}(s-\rho) s^{\frac{d-1}{2}} f(s)  \int_\R e^{i \omega\tau} (1-\chi_{i\omega}(\rho))
\\
&\quad\times(1-\chi_{i\omega}(s))(1+\rho)^{\frac{1}{2}-i\omega}(1-s)^{-\frac{1}{2}+i\omega}\O(\langle\omega\rangle^{-1}) \widehat{\gamma}_1(\rho,s,i\omega) d \omega ds
\\
&:=\int_0^1I_3(\rho,s,\tau)f(s) ds.
\end{align*}

 Here, an application of Lemma \ref{osci3} yields the estimate 
\begin{align*}
|I_3(\rho,s,\tau) f(\rho)|\lesssim \rho^{\frac{3-d}{2}} \langle\tau-\ln(1+\rho)+\ln(1-s)\rangle^{-2} (1+\rho)^{\frac{1}{2}} (1-s)^{-\frac{1}{2}} s^{\frac{d-1}{2}} |f(s)|
  \end{align*}
while from Lemma \ref{osci4} we deduce that
\begin{align*}
|I_3(\rho,s,\tau) f(\rho)|&\lesssim \rho^{\frac{5-d}{2}} \langle\tau-\ln(1+\rho)+\ln(1-s)\rangle^{-2}|\tau-\ln(1+\rho)+\ln(1-s)|^{-1} 
\\
&\quad\times(1+\rho)^{\frac{1}{2}} (1-s)^{-\frac{1}{2}} s^{\frac{d-1}{2}} |f(s)|.
\end{align*} 
An interpolation between those two estimates then shows that
\begin{align*}
|T_3(\tau)f(\rho)|\lesssim &\rho^{\frac{3-d}{2}+\frac{d-3}{4d}} \int_0^1 \langle\tau-\ln(1+\rho)+\ln(1-s)\rangle^{-2}|\tau-\ln(1+\rho)+\ln(1-s)|^{\frac{3-d}{4d}} 
\\
&\times(1+\rho)^{\frac{1}{2}} (1-s)^{-\frac{1}{2}} s^{\frac{d-1}{2}} |f(s)| ds
\\
\lesssim &\rho^{\frac{3-d}{2}+\frac{d-3}{4d}} \int_0^1 \langle\tau+\ln(1-s)\rangle^{-2}|\tau-\ln(1+\rho)+\ln(1-s)|^{\frac{3-d}{4d}} 
\\
&\times(1-s)^{-\frac{1}{2}} s^{\frac{d-1}{2}} |f(s)| ds.
\end{align*}
Consequently, one can derive the desired estimates on $T_3$ by employing the same strategy as for $T_2$.
Similarly, one deduces that
\begin{align*}
|T_4(\tau)f(\rho)|&\lesssim \rho^{\frac{3-d}{2}+\frac{d-3}{4d}} \int_0^1 \langle\tau-\ln(1-\rho)+\ln(1-s)\rangle^{-2}|\tau-\ln(1-\rho)+\ln(1-s)|^{\frac{3-d}{4d}} 
\\
&\quad \times(1-\rho)^{\frac{1}{2}} (1-s)^{-\frac{1}{2}} s^{\frac{d-1}{2}} |f(s)| ds
\end{align*} 
and
as $(1-\rho)^{\frac{1}{2}}\langle\tau-\ln(1-\rho)+\ln(1-s)\rangle^{-2}\lesssim \langle\tau+\ln(1-s)\rangle^{-2}$, we conclude that
\begin{align*}
|T_4(\tau)f(\rho)|&\lesssim \rho^{\frac{3-d}{2}+\frac{d-3}{4d}} \int_0^1 \langle\tau+\ln(1-s)\rangle^{-2}|\tau-\ln(1-\rho)+\ln(1-s)|^{\frac{3-d}{4d}} 
\\
&\quad \times (1-s)^{-\frac{1}{2}} s^{\frac{d-1}{2}} |f(s)| ds.
\end{align*}
The Strichartz estimates for $T_4$ therefore also follow in similar fashion as the ones for $T_2$ and since the remaining kernels can be bounded analogously to $T_3$ and $T_4$, we conclude this proof.
\end{proof}
Unfortunately these estimate alone are not enough to control our semigroup, given that the right-hand side in \eqref{Eq:rep semigroup} also contains the term $\lambda f_1$. To remedy this problem, we introduce operators $\dot{T}_{j,\varepsilon} $ as
\begin{align*}
T_{j,\varepsilon}(\tau)f(\rho):=&\frac{1}{2\pi i}\lim_{N \to \infty}\int_{\varepsilon+iN}^{\varepsilon+iN}
\lambda e^{\lambda \tau}\int_0^1\widetilde{G}_j(\rho,s,\lambda)f(s) ds d\lambda
\end{align*}
and 
\[
\dot{T}_{j}(\tau)f(\rho)=\lim_{\varepsilon \to 0^+} \dot{T}_{\varepsilon,j}(\tau)f(\rho)
\]
for $f\in C^1(\overline{\B^d_1}).$
\\
This additional power of $\lambda$ makes life a bit more tricky for us, as it spoils the absolute convergence of the oscillatory integrals. To get around this difficulty we will perform integrations by parts in the $s$ integral to recover the absolute convergence. Using this trick we will show that the limiting operators $\dot{T}_j$ exist for each $\tau\geq 0$ and $\rho\in (0,1)$ and satisfy Strichartz estimates.
\begin{lem}\label{lem:strichartz 2}
Let $p\in [2,\infty]$ and $q\in [\frac{2d}{d-2},\frac{2d}{d-3}]$ be such that
$\frac{1}{p}+\frac{d}{q}=\frac{d}{2}-1$.
Then the operators $\dot{T}_j$ satisfy
\begin{align*}
\|\dot{T}_j(.)f\|_{L^p(\R_+)L^q(\B^d_1)}\lesssim \|f\|_{H^1(\B^d_1)}
\end{align*}
for all $f\in C^1(\overline{\B^d_1})$.
\end{lem}
\begin{proof}
For $j=1$ we can again without any problems take limits and interchange the order of integration as we can exchange powers of $s$ for decay in $\omega$. Doing this leads to 
\begin{align*}
\dot{T}_1(\tau) f(\rho)&=
\int_\rho^1 f(s) \int_\R e^{i\omega\tau}\chi_{i\omega}(s)(1-\rho^2)^{-\frac{1}{4}-\frac{i\omega}{2}}(1-s^2)^{-\frac{3}{4}-\frac{i\omega}{2}}
\\
&\quad\times\O(\rho^{\frac{3-d}{2}+\frac{1}{8}}s^{\frac{d-3}{2}-\frac{1}{4}}\langle\omega\rangle^{-1-\frac{1}{8}})  d\omega ds
\end{align*}
and we can employ Lemmas \ref{osci1} and \ref{teclem1} to compute that
\begin{align*}
|\dot{T}_1(\tau)f(\rho)|&\lesssim \langle\tau\rangle^{-2}\rho^{\frac{3-d}{2}+\frac{1}{8}}\int_0^1 s^{\frac{d-3}{2}-\frac{1}{4}}f(s) ds
\\
&\lesssim \langle\tau\rangle^{-2}\rho^{\frac{3-d}{2}+\frac{1}{8}} \||.|^{-1}f\|_{L^2(\B^d_1)}
\\
&\lesssim \langle\tau\rangle^{-2}\rho^{\frac{3-d}{2}+\frac{1}{8}} \|f\|_{H^1(\B^d_1)}
\end{align*}
by Lemma \ref{teclem1}.
Consequently, the claimed estimates on $\dot{T}_1$ follow as the ones for $T_1$ in Lemma \ref{lem:strichartz 1}. Analogously, we can bound $\dot{T}_5$ and so, we move on to $\dot{T}_2$, for which we note that we can again take limits, as we can once more exchange powers of $\rho$ for decay in $\omega$. An integration by parts shows that
\begin{align*}
\dot{T}_2(\tau)f(\rho)&=\int_\R  \omega e^{i\omega\tau} \O(\rho^{0}\langle\omega\rangle^{\frac{d-3}{2}})\chi_{i\omega}(\rho)(1-\rho^2)^{-\frac{1}{4}-\frac{i\omega}{2}} \int_\rho^1 (1-\chi_{i\omega}(s))
\\
& \quad\times s^{\frac{d-1}{2}}(1-s)^{-\frac{1}{2}+i\omega}\widehat{\beta}(\rho,s,i\omega) f(s) ds d\omega
\\
&=\int_\R  \omega e^{i\omega\tau} \O(\rho^{0}\langle\omega\rangle^{\frac{d-5}{2}})\chi_{i\omega}(\rho)(1-\rho^2)^{-\frac{1}{4}-\frac{i\omega}{2}}(1-\chi_{i\omega}(\rho))
\\
& \quad\times \rho^{\frac{d-1}{2}}(1-\rho)^{\frac{1}{2}+i\omega}\widehat{\beta}(\rho,\rho,i\omega) f(\rho) d\omega
\\
&\quad+ \int_\R  \omega e^{i\omega\tau} \O(\rho^{0}\langle\omega\rangle^{\frac{d-5}{2}})\chi_{i\omega}(\rho)(1-\rho^2)^{-\frac{1}{4}-\frac{i\omega}{2}} \int_\rho^1 (1-s)^{\frac{1}{2}+i\omega}
\\
& \quad\times \partial_s[ s^{\frac{d-1}{2}}(1-\chi_{i\omega}(s))\widehat{\beta}(\rho,s,i\omega) f(s) ]ds d\omega
\\
&:=B_2(\tau)f(\rho)+I_2(\tau)f(\rho).
\end{align*}
Given that $\widehat{\beta}(.,.i\omega)$ behaves like $\O(\langle\omega \rangle^{-1})$ in $\omega$, we can use Lemma \ref{osci3} to obtain the estimate 
\begin{align*}
|B_2(\tau)f(\rho)|\lesssim \langle\tau\rangle^{-2}\rho |f(\rho)|
\end{align*}
and so, by employing Lemma \ref{teclem2} we see that
\begin{align*}
\|B_2(.)f\|_{L^2(\R_+)L^{\frac{2d}{d-3}}(\B^d_1)}+\|B_2(.)f\|_{L^\infty(\R_+)L^{\frac{2d}{d-2}}(\B^d_1)}\lesssim \|f\|_{H^1(\B^d_1)}.
\end{align*}
For $I_2$ we first remark that if the $s$ derivative hits the cutoff, then we can argue similar as for $\dot{T}_1$ and so focus on the other cases. Let $I^1_2(\tau)f(\rho)$ be term where $f$ gets differentiated. Applying Lemma \ref{osci1} with $\alpha=\frac{1}{4}$ shows that
\begin{align*}
|I^1_2(\tau)f(\rho)|\lesssim& \rho^{\frac{5-2d}{4}}\int_0^1 (1-s)^{\frac{1}{2}}\langle\tau-\frac{1}{2}\ln(1-\rho^2)+\ln(1-s)\rangle^{-2}s^{\frac{d-1}{2}}|f'(s)| ds
\\
\lesssim& \rho^{\frac{5-2d}{4}}\langle\tau\rangle^{-2}\int_0^1 s^{\frac{d-1}{2}}|f'(s)| ds
\lesssim  \rho^{\frac{5-2d}{4}}\langle\tau\rangle^{-2}\|f\|_{H^1(\B^d_1)}
\end{align*}
and so
\begin{align*}
\|I^1_2(.)f\|_{L^\infty(\R_+)L^{\frac{2d}{d-2}}(\B^d_1)}\lesssim \|f\|_{H^1(\B^d_1)}.
\end{align*}
For the other endpoint estimate we can argue similarly as when we bounded $T_2$ and since the remaining cases can be bounded likewise, we move on to $\dot{T}_3$.
For $\dot{T}_3$, we again perform an integration by parts, which recovers the absolute convergence of the $\lambda$ integral, thus enabling us to once more take limits and interchange the order of integration.
More precisely, we deduce that
\begin{align*}
\dot{T}_3(\tau)f(\rho)&= f(\rho)\int_\R \omega e^{i \omega\tau} (1-\chi_{i \omega}(\rho))^2(1+\rho)^{\frac{1}{2}+i \omega}(1-\rho)^{\frac{1}{2}+i\omega}
\O(\langle\omega\rangle^{-2}) \widehat{\gamma}_1(\rho,\rho,i \omega) d\omega
\\
&\quad+\rho^{\frac{1-d}{2}}\int_\rho^1 \int_\R \omega e^{i \omega\tau} (1-\chi_{i \omega}(\rho))(1+\rho)^{\frac{1}{2}+i \omega}(1-s)^{\frac{1}{2}+i\omega}
\\
&\quad\times\partial_s[(1-\chi_{i \omega}(s))s^{\frac{d-1}{2}}f(s) \widehat{\gamma}_1(\rho,s,i\omega) ]d\omega ds
\end{align*}
and one can bound $\dot{T}_3 $ in analogy to $T_2$.
Moreover, since the remaining kernels can also be bounded in similar fashion as either $\dot{T}_2$ or $\dot{T}_3$, we conclude this proof.
\end{proof}
Thanks to these kernel bounds we can now easily establish Theorem \ref{thm:strichartz}.
\begin{proof}[Proof of Theorem \ref{thm:strichartz}]
The estimates for $d=3$ have already been established in \cite{Don17} and so, we only have to verify the claim for $d\geq 4$. 
For $\ff \in C^2\times C^1(\overline{\B^d_1})$ the homogeneous estimates are an immediate consequence of the representation \eqref{Eq:rep semigroup} combined with Lemmas \ref{lem:strichartz 1}, \ref{lem:strichartz 2}, and the Strichartz estimates on $\Sf_0$ proven in Lemma \ref{lem: free strichartz}. Therefore, the general homogeneous estimates follow from a density argument.
The inhomogenous ones can then be established as in the proof of Lemma \ref{lem: free strichartz}.
\end{proof}
\section{Nonlinear Theory}

In this section we turn to the proof of Theorem \ref{thm: stability} and so we restrict ourselves to dimensions $d$ with $3\leq d\leq 6$.  Recall, that the nonlinearity $\Nf$ was defined as 
\begin{align*}
\Nf(\uf)=\begin{pmatrix}
0\\
N(u_1)
\end{pmatrix}
\end{align*}
with
\begin{align*}
N(u)=|c_d+u|^{\frac{4}{d-2}}(c_d+u)-c_d^{\frac{d+2}{d-2}}-\frac{2d+d^2}{4}u.
\end{align*}
\begin{lem}\label{lem:nonlinear esti}
Let $3\leq d\leq 6$. Then the estimates 
\begin{align*}
\|\Nf(\uf)\|_{H^1\times L^2(\B^d_1)}\lesssim \|u_1\|_{L^4(\B^d_1)}^2+\|u_1\|_{L^{\frac{2d+4}{d-2}}(\B^d_1)}^{\frac{d+2}{d-2}}
\end{align*}
and
\begin{align*}
\|\Nf(\uf)-\Nf(\vf)\|_{H^1\times L^2(\B^d_1)}\lesssim &\|u_1-v_1\|_{L^4(\B^d_1)}(\|u_1\|_{L^4(\B^d_1)}+\|v_1\|_{L^4(\B^d_1)})
\\
&+\|u_1-v_1\|_{L^{\frac{2d+4}{d-2}}(\B^d_1)}(\|u_1\|_{L^{\frac{2d+4}{d-2}}(\B^d_1)}^{\frac{4}{d-2}}+\|u_1\|_{L^{\frac{2d+4}{d-2}}(\B^d_1)}^{\frac{4}{d-2}})
\end{align*}
hold for all $\uf,\vf \in H^1\times L^2(\B^d_1)$.
\end{lem}
\begin{proof}
Since $N(0)=N'(0)=0$ we obtain that 
\begin{align*}
|N(x)|\lesssim x^2+|x|^{\frac{d+2}{d-2}}.
\end{align*}
Hence,
\begin{align*}
\|\Nf(\uf)\|_{H^1\times L^2(\B^d_1)}\lesssim \|u_1^2\|_{L^2(\B^d_1)}+\|u_1^{\frac{d+2}{d-2}}\|_{L^2(\B^d_1)}=\|u_1\|_{L^4(\B^d_1)}^2+\|u_1\|_{L^{\frac{2d+4}{d-2}}(\B^d_1)}^{\frac{d+2}{d-2}}.
\end{align*}
For the Lipschitz bound, we infer that
\begin{align*}
|N(x)-N(y)|\lesssim |x-y|(|N'(x)|+|N'(y)|)\lesssim |x-y|(|x|+|x|^{\frac{4}{d-2}}+|y|+|y|^{\frac{4}{d-2}})
\end{align*}
from which we conclude that
\begin{align*}
\|\Nf(\uf)-\Nf(\vf)\|_{H^1\times L^2(\B^d_1)}&\lesssim \|u_1-v_1\|_{L^4(\B^d_1)}(\|u_1\|_{L^4(\B^d_1)}+\|v_1\|_{L^4(\B^d_1)})
\\
&\quad+\|u_1-v_1\|_{L^{\frac{2d+4}{d-2}}(\B^d_1)}(\|u_1^{\frac{4}{d-2}}\|_{L^{\frac{d+2}{2}}(\B^d_1)}+\|v_1^{\frac{4}{d-2}}\|_{L^{\frac{d+2}{2}}(\B^d_1)})
\\
&\lesssim \|u_1-v_1\|_{L^4(\B^d_1)}(\|u_1\|_{L^4(\B^d_1)}+\|v_1\|_{L^4(\B^d_1)})
\\
&\quad+\|u_1-v_1\|_{L^{\frac{2d+4}{d-2}}(\B^d_1)}(\|u_1\|_{L^{\frac{2d+4}{d-2}}(\B^d_1)}^{\frac{4}{d-2}}+\|v_1\|_{L^{\frac{2d+4}{d-2}}(\B^d_1)}^{\frac{4}{d-2}})
\end{align*}
by using Hölder's inequality.
\end{proof}
To finally define a notion of a solution, we first introduce the Strichartz space $\mathcal{X}$ as the completion of $C_c^\infty([0,\infty)\times \overline{\B^d_1})$ with respect to the norm
\begin{align*}
\|\phi\|_{\X}:=\|\phi\|_{L^2(\R_+)L^{\frac{2d}{d-3}}(\B^d_1)}+\|\phi\|_{L^{\frac{d+2}{d-2}}(\R_+)L^{\frac{2d+4}{d-2}}(\B^d_1)}.
\end{align*}
Further, for any $\uf \in \mathcal{H}$ and $\phi\in C^\infty_c([0,\infty)\times \overline{\B^d_1}) $, we define
\begin{align*}
\K_{\uf}(\phi)(\tau):= [\Sf(\tau)\uf]_1+\int_0^\tau [\Sf(\tau-\sigma)\Nf((\phi(\sigma),0))]_1 d \sigma-[e^{\tau}\Cf(\phi,\uf)]_1,
\end{align*}
where $$
\Cf(\phi,\uf):=\Pf\left(\uf+\int_0^\infty e^{-\sigma}\Nf((\phi(\sigma),0))d \sigma\right).
$$ 
Finally, we will also make use of the abbreviation
\[\X_\delta:=\{\phi \in \X: \|\phi\|_\X \leq \delta\}
\]
and at last, we come to our definition of a solution.
\begin{defi}\label{def:solutionstrichartz}
Let
  \[ \Gamma^T:=\{(t,r)\in [0,T)\times [0,\infty): r\leq T-t\}. \]
  We say that $u: \Gamma^T\to \R$ is a Strichartz solution of
\[ \left(\partial_t^2-\partial_r^2-\frac{d-1}{r}\partial_r\right)
  u(t,r)= u(t,r)|u(t,r)|^{\frac{4}{d-2}}\]

if $\phi:=\Phi_1=[\Psi-\Psi_*]_1$, with
\[ \Psi(\tau,\rho):=
  \begin{pmatrix}
    \psi(\tau,\rho) \\ \left(\frac{d-2}{2}+\partial_\tau+\rho\partial_\rho\right)\psi(\tau,\rho)
  \end{pmatrix},\qquad \psi(\tau,\rho):=(Te^{-\tau})^{\frac{d-2}{2}}u(T-Te^{-\tau}, Te^{-\tau}\rho),
\]
  belongs to $\mathcal{X}$ and satisfies
\begin{align*}
\phi=\K_{\Phi(0)}(\phi)
\end{align*}
and $\Cf(\phi,\Phi(0))=\textup{\textbf{0}}$.
\end{defi}
Our modus operandi is now essentially a three step process.
First, we will show that when properly restricted, $\K_\uf$ will satisfy the requirements of the Banach fixed point theorem. 
After that, we will show that by choosing our blowup time appropriately, we can achieve that the correction $\Cf$ vanishes.
Then, we will show uniqueness of solutions in the Strichartz space.
We kick off this program with the following lemma.
\begin{lem}\label{lem:integralbound}
Let $\uf \in \mathcal{H}$ and  $\phi\in C^\infty_c([0,\infty)\times \overline{\B^d_1})$. Then $\K_\uf(\phi)\in \mathcal{X}$. Furthermore, the estimate
\begin{align*}
\|\K_\uf(\phi)\|_{\X}\lesssim \|\uf\|_{\mathcal{H}}+\|\phi\|_{\X}^2+\|\phi\|_{\X}^{\frac{d+2}{d-2}}
\end{align*}
holds for all $\uf \in \mathcal{H}$ and  $\phi\in C^\infty_c([0,\infty)\times \overline{\B^d_1})$.
\end{lem}
\begin{proof}
We first investigate $(\I-\Pf)\K_\uf(\phi)$.
By Theorem \ref{thm:strichartz} and Lemma \ref{lem:nonlinear esti}
we have that 
\begin{align*}
\|(\I-\Pf)\K_\uf(\phi)\|_{\X}&\lesssim \|\uf\|_{\mathcal{H}}
+\int_0^\tau\|N(\phi(\sigma))\|_{L^2(\B^d_1)} d \sigma
\\
&\lesssim \|\uf\|_{\mathcal{H}}
+\int_0^\infty\|\phi(\sigma)\|_{L^4(\B^d_1)}^2+\int_0^\infty\|\phi(\sigma)\|_{L^{\frac{2d+4}{d-2}}(\B^d_1)}^{\frac{d+2}{d-2}} d \sigma
\\
&\lesssim \|\uf\|_{\mathcal{H}}
+\|\phi\|_{\X}^2+\|\phi\|_{\X}^{\frac{d+2}{d-2}}.
\end{align*}
Next, we take a look at $\Pf\K_\uf(\phi)$ and compute that
\begin{align*}
\Pf\K_\uf(\phi)=-\int_\tau^\infty e^{\tau-\sigma}\Pf \Nf((\phi(\sigma),0)) d\sigma.
\end{align*}
Now, as $\Pf$ is a bounded linear operator with $\rg\Pf=\text{span}\{\gf\}$, there exists a unique $\widetilde{\gf}\in \mathcal{H}$, with 
\[\Pf \ff=(\ff,\widetilde{\gf})_{\mathcal{H}}\gf
\]
 for all $\ff \in \mathcal{H}$.
This implies that
$\|[\Pf \ff]_1\|_{L^q(\B_1^6)}\lesssim
|(\ff,\widetilde{\gf})_{\mathcal H}|\|g_1\|_{L^q(\B_1^6)}\lesssim \|\ff\|_{\mathcal H}$ for $1\leq q\leq \infty$
and consequently,
\begin{align*}
\|\Pf\K_\uf(\phi)(\tau)\|_{L^q(\B^d_1)}\lesssim \int_\tau^\infty e^{\tau-\sigma} \|N(\phi(\sigma))\|_{L^2(\B^d_1)} d\sigma.
\end{align*}
So, applying Young's inequality yields
\begin{align*}
\|\Pf\K_\uf(\phi)\|_{\X}&\lesssim \left(\|1_{(-\infty,0)}e^{|.|}\|_{L^2(\R)}+\|1_{(-\infty,0)}e^{|.|}\|_{L^{\frac{d+2}{d-2}}(\R)}\right)\int_0^\infty  \|N((\phi(\sigma))\|_{L^2(\B^d_1)} d\sigma
\\
&\lesssim \|\phi\|_{\X}^2+\|\phi\|_{\X}^{\frac{d+2}{d-2}}.
\end{align*}

\end{proof}
\begin{lem}\label{lem: local lip}
Let $\uf \in \mathcal{H}$. Then the local Lipschitz estimate
\begin{align*}
\|\K_\uf(\phi)-\K_\uf(\psi)\|_{\X}\lesssim \|\phi-\psi\|_{\X}\left(\|\phi\|_{\X}+\|\phi\|_{\X}^{\frac{4}{d-2}}+\|\psi\|_{\X}+\|\psi\|_{\X}^{\frac{4}{d-2}}\right)
\end{align*}
holds
for $\phi, \psi \in C^\infty_c ([0,\infty)\times \overline{\B^d_1})$.
\end{lem}
\begin{proof}
We again start by first investigating $(\I-\Pf)\K_\uf(\phi)-(\I-\Pf)\K_\uf(\psi)$. From Theorem \ref{thm:strichartz} and Lemma \ref{lem:nonlinear esti} we infer that
\begin{align*}
\|(\I-\Pf)\K_\uf(\phi)-(\I-\Pf)\K_\uf(\psi)\|_{\X}&\lesssim \int_0^\tau\|N(\phi(\sigma))-N(\psi(\sigma))\|_{L^2(\B^d_1)} d\sigma
\\
&\lesssim \|\phi-\psi\|_{\X}\left(\|\phi\|_{\X}+\|\phi\|_{\X}^{\frac{4}{d-2}}+\|\psi\|_{\X}+\|\psi\|_{\X}^{\frac{4}{d-2}}\right).
\end{align*}
Further, from the calculations in the proof of Lemma \ref{lem:integralbound} we conclude that
\begin{align*}
\|\Pf[\K_\uf(\phi)(\tau)-\K_\uf(\psi)(\tau)]\|_{L^q(\B^d_1)}\lesssim \int_\tau^\infty e^{\tau-\sigma} \|N(\phi(\sigma))-N(\psi(\sigma))\|_{L^2(\B^d_1)} d\sigma
\end{align*}
again for $1\leq q \leq\infty$. Once more applying Young's inequality then yields
\begin{align*}
\|\Pf[\K_\uf(\phi)-\K_\uf(\phi)]\|_{\X}&\lesssim \left(\|1_{(-\infty,0)}e^{|.|}\|_{L^2(\R)}+\|1_{(-\infty,0)}e^{|.|}\|_{L^{\frac{d+2}{d-2}}(\R)}\right)
\\
&\quad\times\int_0^\infty  \|N(\phi(\sigma))-N(\psi(\sigma))\|_{L^2(\B^d_1)} d\sigma
\\
&\lesssim \|\phi-\psi\|_{\X}\left(\|\phi\|_{\X}+\|\phi\|_{\X}^{\frac{4}{d-2}}+\|\psi\|_{\X}+\|\psi\|_{\X}^{\frac{4}{d-2}}\right).
\end{align*}
\end{proof}
The next two results are now a direct consequence of these two Lemmas.
\begin{lem}
Let $\uf \in \mathcal{H}$. Then the operator $\K_\uf$ extends to an operator from $\X$ to $\X$ such that the estimates
\begin{align*}
\|\K_\uf(\phi)\|_{\X}\lesssim \|\uf\|_{\mathcal{H}}+\|\phi\|_{\X}^2+\|\phi\|_{\X}^{\frac{d+2}{d-2}}
\end{align*}
and
\begin{align*}
\|\K_\uf(\phi)-\K_\uf(\psi)\|_{\X}\lesssim \|\phi-\psi\|_{\X}\left(\|\phi\|_{\X}+\|\phi\|_{\X}^{\frac{4}{d-2}}+\|\psi\|_{\X}+\|\psi\|_{\X}^{\frac{4}{d-2}}\right)
\end{align*}
hold for all
$\phi, \psi \in \X$.
\end{lem}
\begin{lem} \label{lem: solution}
There exist constants $c>0$ and $\delta >0$ such that if $\|\uf\|_{\mathcal{H}}\leq\frac{\delta}{c}$, then there exists a unique $\phi \in \X_\delta$ with $\K_\uf(\phi)=\phi$.
\end{lem}
\begin{proof}
Lemma \ref{lem:integralbound} implies that that $\K_\uf $  maps $\X_\delta$ to $\X_\delta$, provided that $c$ is chosen sufficiently large and $\delta$ sufficiently small. Further, by shrinking $\delta$ if necessary, Lemma \ref{lem: local lip} implies that 
\begin{align*}
\|\K_\uf(\phi)-\K_\uf(\psi)\|_{\X}\leq \frac{1}{2}\|\phi-\psi\|_{\X}
\end{align*}
for $\phi, \psi \in \X_\delta.$
Hence, when restricted to $\X_\delta$, $\K_\uf$ satisfies the conditions of the contraction mapping principle.
\end{proof}
Proceeding as previously stated, we move on to showing that we can get the correction term to vanish by choosing an appropriate blow up time.
To that end, recall that due to the coordinate changes in the first section, the prescribed initial data are given by
\begin{align*}
\phi_1(0,\rho)&=\psi_1(0,\rho)-c_d= T^{\frac{d-2}{2}} u(0,T\rho)-c_d=T^{\frac{d-2}{2}} f(T\rho)-c_d
\\
\phi_2(0,\rho)&=\psi_2(0,\rho)-\frac{d-2}{2}c_d= T^{\frac{d}{2}} \partial_0 u(0,T\rho)-\frac{d-2}{2}c_d= T^{\frac{d}{2}} g(T\rho)-\frac{d-2}{2}c_d.
\end{align*}
Further, as we care about initial data which lies close to $u^1[0]$, we also compute that after our coordinate changes the blowup $u_1$ takes the form 
\begin{align*}
\psi^1_1(\tau,\rho)&=T^\frac{d-2}{2}e^{-\frac{d-2}{2}\tau} u_1(T-Te^{-\tau},Te^{-\tau}\rho)=T^\frac{d-2}{2}e^{-\frac{d-2}{2}\tau}c_d (1-T+Te^{-\tau})^{\frac{2-d}{2}}
\\
\psi^1_2(\tau,\rho)&=\frac{d-2}{2}T^\frac{d}{2}e^{-\frac{d}{2}\tau}c_d (1-T+Te^{-\tau})^{-\frac{d}{2}}.
\end{align*}
To deal with these initial data in a convenient form we introduce the mapping
$$\Uf:[1-\delta,1+\delta]\times  H^1\times L^2(\B^d_{1+\delta}) \to \mathcal{H}$$ as
\begin{align*}
\Uf(T,\vf)(\rho)=(T^{\frac{d-2}{2}}v_1(T\rho),T^{\frac{d}{2}}v_2(T\rho))+(T^{\frac{d-2}{2}}c_d,\frac{d-2}{2}T^{\frac{d}{2}}c_d)-(c_d,\frac{d-2}{2}c_d).
\end{align*}
With this, we can rewrite our initial data as 
\begin{align*}
\phi(0,\rho)=\Uf(T,(f-c_d,g-\frac{d-2}{2}c_d)).
\end{align*}
\begin{lem}\label{lem:blowup}
There exist $M>1$ and $\delta>0$ such that for $\vf \in H^1\times L^2(\B^d_{1+\delta})$ with $\|\vf\|_{H^1\times L^2(\B^d_{1+\delta})}\leq\frac{\delta}{M}$ there exist $\phi \in \X_\delta$ and $T^*\in [1-\delta, 1+\delta]$ with $\K_{\Uf(T^*,\vf)}\phi=\phi$ and $C(\phi,\Uf(T^*,\vf))=0$.
\end{lem}
\begin{proof}
We start by computing that
\begin{align*}
\partial_T (c_dT^{\frac{d-2}{2}},\frac{d-2}{2}c_d T^{\frac{d}{2}})\big|_{T=1}=\frac{d-2}{4}c_d \gf
\end{align*}
and so, we can expand $\Uf(T,\vf)$ as
\begin{align*}
\Uf(T,\vf)=(T^{\frac{d-2}{2}}v_1(T\rho),T^{\frac{d}{2}}v_2(T\rho))+\frac{d-2}{4}c_d(T-1)\gf+(T-1)^2\ff_T
\end{align*}
with $\ff_T \in \mathcal{H} $ such that $\|\ff_T\|_{\mathcal{H}}\lesssim 1$, for $T\in [\frac{1}{2},\frac{3}{2}]$.
This allows us write 
\begin{align*}
(\Uf(T,\vf),\gf)_{\mathcal{H}}=O(\frac{\delta}{M}T^0)+\|\gf\|_{\mathcal{H}}^2 O(\delta T^0)+O(\delta^2 T^0) 
\end{align*}
for $T\in [1-\delta,1+\delta]$.
Denote by $\phi_T\in \X_\delta$ the unique fixed point of $K_{\Uf(T,\vf)}$ which exists, thanks to Lemma \ref{lem: solution}, provided that $M$ is chosen sufficiently large and $\delta$ sufficiently small.
Recall further that
\begin{align*}
\Cf(\phi_T,\Uf(T,\vf)):=\Pf\left(\Uf(T,\vf)+\int_0^\infty e^{-\sigma}\Nf((\phi_T(\sigma),0))d \sigma\right)
\end{align*}
By construction we have that
\begin{align} \label{Eq o-terms}
(\Cf(\phi_T,\Uf(T,\vf)),\gf)_{\mathcal{H}}=O(\frac{\delta}{M}T^0)+O(\delta^2T^0)+\frac{d-2}{4}c_d\|\gf\|^2(T-1)
\end{align}
and moreover the $O$-terms in \eqref{Eq o-terms} are continuous functions of $T$.
Thus, the vanishing of $(\Cf(\phi_T,\Uf(T,\vf)),\gf)_{\mathcal{H}}$ is equivalent to the equation $T-1= F(T)$ for continuous $F$ with $|F(T)|\lesssim \frac{\delta}{M}+\delta^2$. Consequently, if we choose $\delta$ and $M$ accordingly then $1+F$ is a continuous map from $[1-\delta,1+\delta]$ to that same interval and therefore it must have a fixed point.
\end{proof}
We now come to the penultimate result which shows uniqueness of our solutions.
\begin{lem}
Let $\uf \in \mathcal{H}$. Then there exists at most one $\phi \in \X$ with $\K_\uf(\phi)=\phi$ and $\Cf(\phi,\uf)=0$.
\end{lem}

\begin{proof}
  Let $\uf \in \mathcal{H}$ and assume that we are given two fixed
  points $\psi,\phi\in \X$ of $\K_\uf$ with
  $\Cf(\psi,\uf)=\Cf(\phi,\uf)=0$. Let
  $\tau_0\geq 0$ be arbitrary and note that for any $\delta>0$ there
  exists an $N_\delta\in\mathbb N$ such that
\begin{equation}\label{smallnes assumption}
\|\psi\|_{\X(I_n)}+\|\phi\|_{\X(I_n)}\leq \delta,
\end{equation}
for all $n\in \{0,1,2,\dots,N_\delta-1\}$,
 where 
\begin{align*}
\|\phi\|_{\X(I)}:=\|\phi\|_{L^2(I)L^{\frac{2d}{d-3}}(\B^d_1)}+\|\phi\|_{L^{\frac{d+2}{d-2}}(I)L^{\frac{2d+4}{d-2}}(\B^d_1)}
\end{align*}
and $I_n:=[n\frac{\tau_0}{N_\delta}, (n+1)\frac{\tau_0}{N_\delta}]$.
From the calculations done in the proof of Lemma \ref{lem: local lip} we then see that
\begin{align*}
&\quad\|\psi(\tau)-\phi(\tau)\|_{L^{\frac{2d}{d-3}}(\B^d_1)}+\|\psi(\tau)-\phi(\tau)\|_{L^{\frac{2d+4}{d-2}}(\B^d_1)}
\\
&\lesssim \left\|\int_0^\tau e^{\tau-\sigma}[N(\phi(\sigma))-N(\psi(\sigma))] d\sigma\right\|_{L^2(\B^d_1)}.
\end{align*}
Thus,
\begin{align*}
\|\psi-\phi\|_{\X(I_0)}
&\lesssim \left\|\int_0^\tau e^{\tau-\sigma} \|N(\phi(\sigma))-N(\psi(\sigma))\|_{L^2(\B^d_1)}d \sigma\right\|_{L^2_\tau(I_0)}
\\
&\quad+\left\|\int_0^\tau e^{\tau-\sigma} \|N(\phi(\sigma))-N(\psi(\sigma))\|_{L^2(\B^d_1)}d \sigma\right\|_{L^{\frac{d+2}{d-2}}_\tau(I_0)}
\\
&\lesssim  e^{\tau_0} \int_0^{\tau_0/N_\delta}\|N(\phi(\sigma))-N(\psi(\sigma))\|_{L^2(\B^d_1)}d \sigma
\\
&\lesssim 
\|\phi-\psi\|_{\X(I_0)}\left(\|\phi\|_{\X(I_0)}+\|\phi\|_{\X(I_0)}^{\frac{4}{d-2}}+\|\psi\|_{\X(I_0)}+\|\psi\|_{\X(I_0)}^{\frac{4}{d-2}}\right).
\end{align*}
By choosing $\delta$ sufficiently small, we see that
\begin{align*}
\|\phi-\psi\|_{\X(I_0)}\leq \frac{1}{2} \|\phi-\psi\|_{\X(I_0)}
\end{align*}
and so $\|\phi-\psi\|_{\X(I_0)}=0$.
Continuing inductively we therefore conclude that
\begin{align*}
\|\phi-\psi\|_{\X([0,\tau_0])}=0
\end{align*}
and as $\tau_0>0$ was chosen arbitrarily, the claim follows.
\end{proof}

Now, we are finally able to prove Theorem \ref{thm: stability}.
\begin{proof}[Proof of Theorem \ref{thm: stability}]
Let $M>1$ be sufficiently large and $\delta>0$ sufficiently small. Further,  for $\vf= (f,g)-u^1[0]$ with
 $$\|\vf\|_{H^1\times L^2(\B^d_{1+\delta})}\leq \frac{\delta}{M},$$ 
 let $u$ be the corresponding solution, $\phi$ be the associated function in similarity coordinates, and $T$ be the corresponding blowup time from Lemma \ref{lem:blowup}. Then
\begin{align*}
\delta^2&\geq\|\phi\|_{L^2(\R_+)L^{\frac{2d}{d-3}}(\B^d_1)}^2=\int_0^\infty\|\psi(\sigma,.)-c_d\|_{L^{\frac{2d}{d-3}}(\B^d_1)}^2 d\sigma
\\
&= \int_0^T(T-t)^{-1} \|\psi_1(-\log(T-t)+\log(T),.)-c_d\|_{L^{\frac{2d}{d-3}}(\B^d_1)}^2 dt
\\
&= \int_0^T \|(T-t)^{\frac{2-d}{2}}\psi_1(-\log(T-t)+\log(T),\frac{.}{T-t})-c_d(T-t)^{\frac{2-d}{2}}\|_{L^{\frac{2d}{d-3}}(\B^d_{T-t})}^2 dt
\\
&= \int_0^T\|u(t,.)-u^T(t,.)\|_{L^{\frac{2d}{d-3}}(\B^d_{T-t})}^2 dt.
\end{align*}
\end{proof}

\bibliography{references}
\bibliographystyle{plain}

\end{document}